\documentclass[reqno]{amsart}
\usepackage{hyperref}
\usepackage{amssymb}
\numberwithin{equation}{section}
\newtheorem{theorem}{Theorem}[section]
\newtheorem{lemma}[theorem]{Lemma}
\newtheorem{proposition}[theorem]{Proposition}
\newtheorem{definition}[theorem]{Definition}

\newtheorem{remark}[theorem]{Remark}

\def\R{\mathbb R}

\begin{document}
\title
[\hfil Almost periodicity]
{On the almost periodicity of nonautonomous evolution equations and application to Lotka--Volterra systems}
\author[K. Khalil]
{Kamal Khalil}

\address{Kamal Khalil\newline
 Department of Mathematics, Faculty of Sciences Semlalia, Cadi Ayyad University\\
 B.P. 2390, 40000 Marrakesh, Morocco.}
\email{kamal.khalil.00@gmail.com}


\keywords{Semilinear evolution equations; almost periodicity; nonautonomous  \hfill\break\indent  evolution equations; exponential dichotomy; Lotka--Volterra predator--prey models}

\begin{abstract}
Consider the nonautonomous semilinear evolution equation of type:  $(\star) \;  u'(t)=A(t)u(t)+f(t,u(t)), \; t \in \mathbb{R},$ where $ A(t), \ t\in \mathbb{R} $ is a family of closed linear operators in a Banach space $X$, the nonlinear term $f$, acting on some real interpolation spaces, is assumed to be almost periodic just in a weak sense (i.e. in Stepanov sense) with respect to $t$ and Lipschitzian in bounded sets with respect to the second variable. We prove the existence and uniqueness of almost periodic solutions in the strong sense (Bohr sense) for equation $ (\star)  $ using the exponential dichotomy approach. Then, we establish a new composition result of Stepanov almost periodic functions by assuming just the continuity of $f$ in the second variable. Moreover, we provide an application to a nonautonomous system of reaction-diffusion equations describing a Lotka--Volterra predator--prey model with diffusion and time-dependent parameters in a generalized almost periodic environment. 
\end{abstract}

\maketitle
\section{Introduction}
The concept of almost periodicity introduced by H. Bohr in \cite{H.Bohr} is a trivial generalization of the well-known periodicity. Both conecpts are given in the literature by means of continuous bounded functions. However, W. Stepanov in \cite{Step} introduced a more general definition of almost periodicity, this time, it is through locally integrable functions with certain weak appropriate boundedness conditions. Hence, an almost periodic function in Stepanov sense is not necessarily continuous or bounded. The motivation for considering almost periodic functions in Stepanov sense is the existence of a more wide class of unbounded and discontinuous functions describing periodic and almost periodic motions, see Section \ref{Section5} for more details about examples from this class.\\

Now, consider the following nonautonomous semilinear evolution equation,
\begin{equation}
     x'(t)= A(t)x(t) + f(t, x(t))\quad \text{for}\; t \in \mathbb{R},  \label{Eq_1}    
\end{equation}
where $ (A(t),D(A(t))) $, $t\in \mathbb{R}$  is a family of linear closed operators defined on a Banach space $ X $ (not necessarily densely defined) that generates an analytic evolution family $ (U(t,s))_{t\geq s} $ which has an exponential dichotomy on $\mathbb{R}$. The nonlinear term $ f: \mathbb{R}\times X^{t}_{\alpha} \longrightarrow X $ is just almost periodic in Stepanov sense (i.e. in the weak sense) of order $ 1\leq p <\infty $ with respect to $t$ and Lipschitzian in bounded sets with respect to the second variable, where $ X^{t}_{\alpha},\,  0<\alpha <1 $, $ t\in \mathbb{R} $ are some continuous interpolation spaces with respect to the linear operators $ (A(t),D(A(t))) $, $t\in \mathbb{R}$, such that the following continuous embedding hold
$$
D(A(t)) \hookrightarrow  X^{t}_{\beta}  \hookrightarrow  X^{t}_{\alpha} \hookrightarrow  X \quad \text{ for all }  0 < \alpha < \beta < 1, \, t\in \mathbb{R}. $$  
In this work, we establish sufficient weak conditions on the forcing terms $ A(\cdot) $ and $ f $ insuring the existence and uniqueness of almost periodic solutions (in the strong sense) to equation \eqref{Eq_1}. Moreover, we introduce a new composition result of Stepanov almost periodic functions of order $1\leq p <\infty$. That is, for given any function $f:\mathbb{R}\times Y \longrightarrow X $ which is Stepanov almost periodic with respect to $t$ and continuous with respect to $x$  and  $u:\mathbb{R}\longrightarrow Y$ is Stepanov almost periodic for any Banach spaces $X$ and $Y$. Then, $ f(\cdot,u(\cdot)) $ is also Stepanov almost periodic. Our composition result require just the continuity of $f$ in the second argument while the literature's works assume the uniform Lipschitz condition which is very stronger as hypothesis than ours, see \cite{AkdEss,Moi,Ding,Li}. Technically, our strategy concerns to study at first the following linear inhomogeneous equation: 
\begin{equation}
 u'(t)=A(t)u(t)+h(t), \quad   t\in\mathbb{R}, \label{Eq_2}
 \end{equation}
where $ h $ is almost automorphic in Stepanov sense of order $1\leq p <\infty$. We show that the unique mild solution given by: 
\begin{equation*}
u(t)=\int_{\mathbb{R}}G (t,s)h(s)ds, \quad  t\in\mathbb{R},
\end{equation*}
where $ G(\cdot,\cdot) $ is the associated Green function, is almost periodic in the strong sense. Hence, by our suitable composition result, we prove the results to equation \eqref{Eq_1} using a fixed point argument. \\

Furthermore, an illustrating application to a nonautonomous system of reaction--diffusion equations describing a Lotka--Volterra  predator--prey model with diffusion and time-dependent parameters in a generalized almost periodic environment is provided. In fact, we study the existence and uniqueness of almost periodic solutions to the following model: 
\begin{equation}
 \left\{
    \begin{aligned}
     \frac{\partial}{\partial t}u(t,x)&=  d_1 (t)\Delta u(t,x) + a(t)u(t, x)- c_1 (t)\dfrac{ v(t,x) u(t,x) }{1+| \nabla v(t,x)|}, \, t \in \mathbb{R}, x\in \Omega, \\   
     \frac{\partial}{\partial t}u(t,x)&=  d_2 (t)\Delta v(t,x) - b(t)v(t, x)+ c_2 (t)\dfrac{ u(t,x) v(t,x)}{1+|\nabla u(t,x)|}, \,t \in \mathbb{R}, x\in \Omega, \\
     u(t,x)|_{\partial \Omega}& = 0 ; \  v(t,x)|_{\partial \Omega} =0, \; t \in \mathbb{R}, x\in \partial\Omega  \label{Eq3Chapter1}
    \end{aligned}
  \right. 
\end{equation}
in a bounded domain (a habitat) $\Omega \subset \R^N$ ($N \geq 1$) with Lipschitz boundary $ \partial \Omega $, provided that the time parameters $ d_i $, $ a$, $b$ and $c_i$, $i=1,2$, are just almost periodic in Stepanov sense, where $ d_i $ are the diffusion terms of prey and predator populations, $a$ and $b$ are respectively the growth terms without interaction of the populations. Besides, the (nonlinear) interaction terms mean that, the predators consume preys in proportion to the number of the (space) variation of the predators with the proportion $c_{1} > 0$, while the predator population grows proportionally to the variation of small preys (they can consume the proportion $c_{2} > 0$). For more details about the study of models of type \eqref{Eq3Chapter1}, we refer to \cite{Cantrell,Hetzer,Langa,AAlaoui,Perthame} and references therein.  Therefore, we consider the product Banach space $X:=C_{0}(\overline{\Omega})\times C_{0}(\overline{\Omega}) $ and we prove, under suitable considerations, that our model \eqref{Eq3Chapter1} is equivalent to equation  \eqref{Eq_1}. Then, we establish our results.  \\

In the literature, we found several works devoted to the existence and uniqueness of almost periodic solutions to equation \eqref{Eq_1}, see \cite{AkdEss,Moi,BHR,Ding,GR,Li,Man-Sch,Vu}. In \cite{BHR,GR,Vu} the authors proved the existence and uniqueness of almost periodic solutions to equation \eqref{Eq_2} (i.e. equation \eqref{Eq_1} in the linear inhomogeneous case) where $ A(\cdot) $ is periodic and $h$ is almost periodic in the strong sense. Moreover, in \cite{Man-Sch} the authors proved that equation \eqref{Eq_2} has a unique almsot periodic solution provided that $ R(\omega,A(\cdot)) $, for some $ \omega \in \mathbb{R} $, and $h$ are almost periodic in the strong sense. Furthermore, in \cite{AkdEss,Moi,Ding,Li} the authors proved the existence and uniqueness to equation \eqref{Eq_1} where $f$ is Stepanov almost periodic with respect to $t$ and globally Lipschitzian in the second variable.  \\

The organization of this paper is as follows: Section \ref{Section2} is devoted to preliminaries, notations and main hypotheses of this work. In Section \ref{Section3} we give our new composition result of Stepanov almost periodic functions of order $1\leq p<\infty$. Section \ref{Section4} is provided to the existence and uniqueness results of almost periodic solution to equations  \eqref{Eq_2} and  \eqref{Eq_1} respectively. In Section \ref{Section5}, we study the existence and uniqueness of almost periodic solutions to our model \eqref{Eq3Chapter1}.

%

\section{Notations, Preliminaries and main hypotheses}\label{Section2}

In this section, we recall notations, definitions and preliminary results needed in the following. 
\subsection*{Notations}
Throughout this work, $(X,\|\cdot\|) $ and $(Y,\|\cdot\|_{Y})$ are two any Banach spaces and $(\mathcal{L}(X),\|\cdot\|_{\mathcal{L}(X)})$ is the Banach space of bounded linear operators in $ X $. $BC(\mathbb{R},X)$ equipped with the sup norm, denoted by $ \| \cdot \|_{\infty} $, is the Banach space of bounded continuous functions $f$ from $\mathbb{R}$ into $ X$. Moreover, for $ 1\leq p <\infty $, $ q $ denotes its conjugate exponent defined by $ \dfrac{1}{p} +\dfrac{1}{q}=1$ if $ p\neq  1$ and $ q=\infty $ if $ p=1 $. By  $ L^{p}_{loc}(\mathbb{R},X)$ (resp. $ L^{p}(\mathbb{R},X)$), we designate the space (resp. the Banach space) of all equivalence classes of measurable functions $f$ from $\mathbb{R}$ into $ X$  such that $\|f(\cdot)\|^{p}$ is locally integrable (resp. integrable). We denote by $\Gamma(\cdot)$ the gamma function defined by $ \Gamma(z):= \displaystyle \int_{0}^{\infty} s^{z-1} e^{-s} ds  $ for $z>0$. Moreover, in the following it is assumed that the resolvent set $ \rho(A) $ of $(A,D(A))$ is defined by 
$ \rho(A):= \lbrace \lambda \in \mathbb{C}: \, (\lambda -A)^{-1}\; \text{exists in}\; \mathcal{L}(X)\rbrace \subset \mathbb{C}$
and the spectrum $ \sigma(A) $ of $(A,D(A))$ is defined as: $ \sigma(A):= \mathbb{C}\setminus \rho(A) $. For $\lambda \in \rho(A)$, the resolvent operator $R(\lambda, A)$ is defined by $ R(\lambda, A) := (\lambda - A)^{-1}.  $
\subsection*{Evolution families and intermediate spaces}
Let $(A(t),D(A(t))), \; t\in \mathbb{R}$ be a family of linear closed  operators on a Banach space $X$ that satisfies the following conditions: there exist constants  $ \omega \in \mathbb{R}, \theta  \in (\frac{\pi}{2},\pi ), M>0$ and $ \eta, \nu \in (0,1] $ with $ \eta +\nu >1 $ such that 
\begin{equation}
\left\{
    \begin{aligned}
  &    \Sigma_{\omega,\theta}:= \lbrace z \in \mathbb{C}: z \neq 0 , \mid arg(z  ) \mid \leq \theta   \rbrace \subset \rho(A(t)-\omega) \\
  &    \| \lambda R(\lambda,A(t)-\omega)  \|_{\mathcal{L}(X)} \leq L,         \label{AquTerCon1}
    \end{aligned} 
  \right. 
\end{equation}
\begin{eqnarray}
\|(A(t)-\omega) R(\lambda,A(t)-\omega)[R(\omega,A(t))-R(\omega,A(s))  ]\|_{\mathcal{L}(X)} \leq \dfrac{M |t-s |^{\eta}}{|\lambda |^{\nu}} \label{AquTerCon2}
\end{eqnarray}
for all $ t \geq s,\, t,s \in \mathbb{R} $ and $ \lambda \in \Sigma_{\omega,\theta}$. The domains $ D(A(t)) $ of the operators $ A(t) $ may change with respect to $t$ and do not required to be dense in $X$.  The condition \eqref{AquTerCon1} means that each operator $ A(t) $ generates a bounded analytic semigroup $  (T_{t}(s))_{s\geq 0} $ where the semigroups may be not strongly continuous at $ 0 .$  The condition \eqref{AquTerCon2} provides some stability in the dependence on $t$ of the operators $ A(t) $. Note that, the above conditions \eqref{AquTerCon1} and \eqref{AquTerCon2} was introduced in \cite{AquTer1} to solve the following Cauchy problem 
\begin{align}
 \left\{
    \begin{array}{ll}
        u'(t)=A(t)u(t), & t\geq s \\
       u(s)=x \in X, & .
    \end{array}  \label{Eq Non Aut Inho}
\right.  
\end{align}
in the parabolic context, where the unique solution (in the classical sense) is given by a two-parameter family of bounded linear operators in $X$, i.e., $u(t)=U(t,s)x$ with $(U(t,s))_{t\geq s} \subset \mathcal{L}(X)$ is an evolution family. More precisely, for $ t>s $ the map $ (t,s)\longmapsto U(t,s) \in \mathcal{L}(X) $ is continuous and continuously differentiable in $t$, $ U(t,s) $  maps $ X $ into $ D(A(t)) $ and it holds that $ \dfrac{\partial U(t,s)}{\partial t}=A(t)U(t,s) .$ Moreover, $ U(t,s) $ and $ (t-s)A(t)U(t,s) $ are exponentially bounded and that satisfy 
$$ U(t,s)U(s,r)=U(t,r) \quad \mbox{and}\quad U(t,t)=I \quad \mbox{for} \quad t\geq s\geq r. $$
Furthermore, for $ s\in \mathbb{R} $ and $ x \in \overline{D(A(s))} $, the unique solution $ t\longmapsto u(t)=U(t,s)x $ in $ C([s,\infty),X)\cap C^{1}((s,\infty),X) $ is continuous at $ t=s $. For more details, we refer to \cite{AquTer1,EN,Lun}. \\

Now, we introduce the interpolation spaces for  the operators $A(t), \ t\in \mathbb{R}$. Let $A$ be a sectorial operator, i.e., $ A $ satisfy \eqref{AquTerCon1} instead of $ A(t) $ (it is well known that $ A $ generates an analytic semigroup $(T_{A}(t))_{t\geq 0}$ on $X$). For $ \alpha  \in (0,1), $ we give the real interpolation spaces: $$ X_{\alpha}:= \overline{D(A)} ^{\|\cdot \|_{\alpha}},  \quad \mbox{where}\quad \|x \|_{\alpha}:= \sup_{\lambda >0} \| \lambda^{\alpha}(A-\omega) x \| \quad \mbox{for all}\; x \in D(A)  .$$  Then, $ (X_{\alpha}, \|\cdot \|_{\alpha}) $ are Banach spaces. Let $ X_{0}:= X ,$ $ X_{1}:= D(A) $ and $ \|x\|_{0}=\|x\|,$ $ \|x\|_{1}=\|(A-\omega) x\| $ be the corresponding norms respectively. 
\begin{eqnarray}
D(A) \hookrightarrow  X_{\beta}  \hookrightarrow  X_{\alpha} \hookrightarrow  X  \label{embiddings}
\end{eqnarray}
for all $ 0 < \alpha < \beta < 1.  $ Let $ A(t), $ $ t\in \mathbb{R} $ which satisfies \eqref{AquTerCon1}, we set $ X_{\alpha }^{t}:=X_{\alpha } $
where $ A(t) $ is taken instead of $ A $ in the definition of the spaces $ X_{\alpha} $, $ 0 < \alpha <  1 $, and the corresponding norms. Then, the embeddings in \eqref{embiddings} hold with uniformly bounded norms in $t \in \mathbb{R}$. Moreover, in the case of a constant domain, i.e., $D:= D(A(t))$, $t\in \mathbb{R}$, we can replace assumption \eqref{AquTerCon2} with the following:\\

There exist constants $\omega \in \mathbb{R} $, $L \geq 0$ and $0 < \mu \leq 1$ such that
\begin{eqnarray}
\| (A(t)-A(s))R(\omega ,A(r)) \| \leq L |t-s | ^{\mu} \quad \text{for } t,s,r \in \mathbb{R}. \label{AquTerCon21}
\end{eqnarray}

A sufficient condition ensuring \eqref{AquTerCon21} is the next:\\
\begin{eqnarray}
\| (\omega - A(t))R(\omega ,A(s))-I_X  \| \leq L_0 |t-s|^{\mu_{0}} \quad \text{for } t,s \in \mathbb{R}  \label{AquTerCon22}
\end{eqnarray}
for some $\omega \in \mathbb{R} $, $L_0 \geq 0$ and $0 < \mu_0 \leq 1$.\\

Now, we introduce exponential dichotomy of an evolution family which is an important tool in our  study, see \cite{AkdEss,EN,Schnaubelt} 
\begin{definition}\cite{AkdEss}
An evolution family $ (U(t,s))_{s\leq t} $ on a Banach space $X$ is called has an exponential dichotomy (or hyperbolic) in $ \mathbb{R} $ if there exists a family of projections $P(t) \in \mathcal{L}(X)$, $t \in \mathbb{R}$, being strongly continuous with respect to $t$, and constants  $\delta ,N > 0$ such that
\begin{itemize}
\item[\textbf{(i)}] $ U(t,s) P(s)=P(t)U(t,s). $
\item[\textbf{(ii)}] $U(t,s): Q(s)X\longrightarrow Q(t)X$ is invertible with the inverse $ \tilde{U}(t,s) $ (i.e., $ \tilde{U}(t,s)=U(s,t)^{-1} $).
\item[\textbf{(iii)}] $\|U(t,s)P(s)\| \leq N e^{-\delta (t-s)}$ and $\| \tilde{U}(s,t)Q(t)\| \leq N e^{-\delta (t-s)}$
\end{itemize} 
for all $t,s \in \mathbb{R}$ with $s \leq t$, where, $Q(t) := I-P(t) .$
\end{definition}
Hence, given a hyperbolic evolution family  $(U(t,s))_{s\leq t} $, then its associated Green function is 
defined by:
\begin{equation}
G(t,s)=\left\{
    \begin{aligned}
     U(t,s) P(s),&  \quad t,s \in \mathbb{R},\; s \leq t,\\   
     -\tilde{U}(t,s) Q(s), &  \quad t,s \in \mathbb{R},\; s > t .\\
    \end{aligned} \label{GreenFun}
  \right. 
\end{equation}
Next, we show that exponential dichotomy can be characterized in many cases, for more details, see \cite{He}.
\begin{remark}\label{Remark1 Examples Gfun}
\textbf{(a)}  If  $U(t,s)=T(t-s)$, $t\geq s$  and  $ (T(t))_{t\geq 0} \subset \mathcal{L}(X)$  is a semigroup of bounded linear operators. In that case, the generator of $ (U(t,s))_{t\geq s} $ is a time constant sectorial operator $ (A,D(A)) $ such that 
 $\sigma (A) \cap \{ \lambda \in \mathbb{C}: \ | Re(\lambda) | \leq \beta \} =\emptyset $ for some $\beta >0$. 
Then $ (U(t,s))_{t\geq s} $ has an exponential dichotomy on $\mathbb{R}$ with constant projections $P(t)=P$ and $Q(t):=Q$, $ t\in \mathbb{R}$ and exponent $\delta:=\beta$.\\

\textbf{(b)}  If $ (A(t))_{t\in \mathbb{R}} $  is $p$-periodic (with $p>0$). In that case the evolution family $ (U(t,s))_{t\geq s} $ is also $p$-periodic in $t$ and $s$ and the spectrum of  'the period map' $ \mathbb{U}_p:= U(t+p,t), \, t\in \mathbb{R}  $ is independent of $t$.  Moreover if $\sigma (\mathbb{U}_p) \cap \{ \lambda \in \mathbb{C}: e^{-p \beta} \leq   | \lambda | \leq e^{p \beta}  \} =\emptyset $, for some $\beta >0$. 
Then $ (U(t,s))_{t\geq s} $ has an exponential dichotomy on $\mathbb{R}$ with $p$-periodic projections $P(t)$ and $Q(t)$, $ t\in \mathbb{R}$ and exponent $\delta:=\beta$.\\

\textbf{(c)}  If  $ (U(t,s))_{t\geq s} $ is exponentially stable i.e., $  \omega(U) <0$, where  $$ \omega(U):= \inf \{ \omega  \in \mathbb{R} : \exists \; M_{\omega} \geq 1 \text{ with } \;  \|U(t,s)\| \leq M_{\omega} e^{\omega (t-s)},  \; t\geq s, \ s \in \mathbb{R}     \} .$$ 

Then $ (U(t,s))_{t\geq s} $ has an exponential dichotomy on $\mathbb{R}$ with projections $P(t)=I_X, \, t\in \mathbb{R}$, (the identity operator of $X$) and exponent $\delta:=-\omega(U)$.
\end{remark}

More generally, from \cite{Schnaubelt}, the exponential dichotomy holds if the following is true: \\

Assume \eqref{AquTerCon21} holds and the semigroups $(T^{t}(\tau))_{\tau \geq 0}$ are hyperbolic with projections $P_t$ and constants $N, \delta >0$ such that 
$ \| A(t) T^{t}(\tau)P_t \| \leq \psi (\tau) $ and  $ \| A(t)T_{Q}^{t}(\tau)Q_t\| \leq \psi(-\tau) $  for  $ \tau >0$  and a function $ \psi$ such that the mapping $\mathbb{R} \ni s \longmapsto \varphi(s):= |s |^{\mu} \psi (s)$  is integrable with  $ L \| \varphi \|_{L^{1}(\mathbb{R})} <1$ .\\

Now, we give some dichotomy estimates of the evolution family $ (U(t,s))_{s\leq t} $ through the interpolation spaces $ X_{\alpha}^{t}, \; 0 \leq \alpha \leq 1 $, see \cite{Br2} and \cite{Moi2} for a more general result.  
\begin{theorem}\label{prop2.3}
Let $ x\in X, \, 0<\alpha \leq 1$. Then, the following hold:
\begin{itemize}
\item[\textbf{(i)}] There exists a constant $ c(\alpha ),$ such that 
\begin{eqnarray}
\|\tilde{U}(t,s) Q(t)x \|^{s}_{\alpha} &\leq & c(\alpha )  e^{-\delta (s-t)} \| x \| \quad \mbox{for}\;  t < s. \label{prop2.3 estim1}
\end{eqnarray}
\item[\textbf{(ii)}] There exists a constant $ m(\alpha ),$ such that 
\begin{eqnarray}
\| U(t,s) P(s)x\|^{t}_{\alpha} &\leq & m(\alpha )(t-s)^{\alpha -1} e^{- \gamma (t-s)} \| x \| \quad \mbox{for}\; t > s. \label{prop2.3 estim2}
\end{eqnarray}
\end{itemize}
\end{theorem}
\subsection{Almost periodic functions}

In this section, we recall important properties on almost periodic functions in the sense of Bohr and that of Stepanov, \cite{Amerio,H.Bohr,Step}. 
\begin{definition}[H. Bohr \cite{H.Bohr}]
A continuous function $ f:\mathbb{R}\longrightarrow X $ is said to be almost periodic, if for every $\varepsilon > 0,$ there exists $ l_{\varepsilon}  > 0 $, such that for every $ a\in \mathbb{R} $, there exists $ \tau \in \left[a,a+ l_{\varepsilon} \right]  $ satisfying: $$ \| f(t+\tau)-f(t) \| < \varepsilon \quad \mbox{for all}\; t\in \mathbb{R} .$$ The space of all such functions is denoted by $ AP(\mathbb{R},X) .$
\end{definition}
\begin{theorem}[S. Bochner \cite{Boch2}]\label{Theorem charac AP with sequences}
A continuous function $ f:\mathbb{R}\longrightarrow X $ is  almost periodic if and only if for every sequence $(\sigma_{n})_{n \geq 0}$ of real numbers, there exists a subsequence $(s_{n})_{n \geq 0}\subset (\sigma_{n})_{n \geq 0} $ and a continuous function $ g :\mathbb{R}\longrightarrow X$, such that  
\begin{eqnarray}
 g(t)=:\lim_{n} f(t+s_{n}) \; \text{uniformly on } t\in \mathbb{R}. \label{charac AP with sequences}
\end{eqnarray}
\end{theorem}
\begin{definition}[S. Bochner \cite{Boch2}]
A continuous function $ f:\mathbb{R}\longrightarrow X $ is said to be almost automorphic if for every sequence $(\sigma_{n})_{n \geq 0}$ of real numbers, there exist a subsequence $(s_{n})_{n \geq 0}\subset (\sigma_{n})_{n \geq 0} $ and a function $g \in L^{\infty}(\mathbb{R},X)$, such that the following limits
$$ g(t)=:\lim_{n} f(t+s_{n})\quad \mbox{and}\quad f(t)=\lim_{n} g(t-s_{n})  $$ 
are well-defined for each $ t\in \mathbb{R}  .$ 
The space of all such functions will be denoted $ AA(\mathbb{R},X) .$
\end{definition}
\begin{definition}
Let $(Z,\|\cdot \|_{Z})$ be any Banach space. A continuous function $F:\mathbb{R}\times\mathbb{R} \rightarrow Z$  is said to be bi-almost periodic if for every $\varepsilon > 0,$ there exists $ l_{\varepsilon}  > 0 $, such that for every $ a\in \mathbb{R} $, there exists $ \tau \in \left[a,a+ l_{\varepsilon} \right]  $ satisfying: $$ \| F(t+\tau,s+\tau)-F(t,s) \|_{Z} < \varepsilon \quad \mbox{for all}\; t,s \in \mathbb{R} .$$ The space of all such functions is denoted by $ bAP(\mathbb{R},X) .$
\end{definition} 
\begin{definition} Let $ 1\leq p< \infty $.  A function $ f\in L^{p}_{loc} (\mathbb{R},X) $ is said to be bounded in Stepanov sense if $$ \displaystyle \sup_{t\in \mathbb{R}}  \left( \int_{\left[ t,t+1\right] } \|f(s)\|^{p}ds\right) ^{\frac{1}{p}}=\displaystyle \sup_{t\in \mathbb{R}}  \left( \int_{\left[ 0,1\right] } \|f(t+s)\|^{p}ds\right) ^{\frac{1}{p}} <\infty.$$ 
The space of all such functions is denoted by $ BS^{p} (\mathbb{R},X)$ and it is provided with the following norm:
\begin{eqnarray*} 
   \|f\|_{BS^{p} }&:=& \displaystyle \sup_{t\in \mathbb{R}}  \left( \int_{\left[ t,t+1\right] } \|f(s)\|^{p}ds\right) ^{\frac{1}{p}} \\ & =& \sup_{t\in \mathbb{R}}\|f(t+\cdot)\|_{L^{p} (\left[ 0,1\right],X)} .
\end{eqnarray*}  
\end{definition}
\begin{definition}[Bochner transform] \label{Definition Bochner transform}
Let $f\in L^{p}_{loc} (\mathbb{R},X) $ for $ 1\leq p< \infty $. The Bochner transform of $f$ is the function $f^b :\mathbb{R}\longrightarrow L^{p} (\left[ 0,1\right],X) $ defined by $$ (f^b (t))(s)= f(t+s) \quad \mbox{for} \; s\in \left[ 0,1\right], \; t\in \mathbb{R}.$$
\end{definition}

 Now, we give the definition of almost periodicity in Stepanov sense.
\begin{definition}\label{Definition APSp} Let $ 1\leq p< \infty $. A function $ f \in L^{p}_{loc} (\mathbb{R},X) $ is said to be almost periodic in the sense of Stepanov (or $S^p$-almost periodic), if for every $\varepsilon > 0,$ there exists $ l_{\varepsilon}  > 0 $, such that for every $ a\in \mathbb{R} $, there exists $ \tau \in \left[a,a+ l_{\varepsilon} \right]  $ satisfying $$ \displaystyle \left( \int_{\left[ t,t+1\right] }\| f(s+\tau)-f(s) \|^{p}ds\right)^{\frac{1}{p}}  < \varepsilon \quad \mbox{for all}\; t\in \mathbb{R}.$$ The space of all such functions is denoted by $ APS^{p}(\mathbb{R},X) .$ 
\end{definition}
\begin{remark}\label{RemCompSpAp}

\textbf{(a)} Every (Bohr) almost periodic function is $S^p$-almost periodic for $ 1\leq p< \infty $. The converse is not true in general (see Proposition \ref{PropCovApS}).\\

\textbf{(b)} For all $ 1\leq p_1 \leq p_2 < \infty $, if $f$ is $S^{p_2}$-almost periodic, then $f$ is $S^{p_1}$-almost periodic. \\

\textbf{(c)} The Bochner transform of an $X $-valued function is a $L^{p} (\left[ 0,1\right],X) $-valued function. Moreover, a function $f $ is $S^p$-almost periodic if and only if $f^{b}$ is (Bohr) almost periodic. 
\end{remark} 
Using the Bochner transform i.e. Definition \ref{Definition Bochner transform} and Theorem \ref{Theorem charac AP with sequences}, we can deduce easily the following characterization of Stepanov almost periodicity using sequences.
\begin{theorem}\label{Theorem charac APSp with sequences} Let  $f\in L^{p}_{loc} (\mathbb{R},X) $. The function $f$ is $S^p $-almost periodic if and only if for every sequence $(s_n )_n $ of real numbers there exists a subsequence $(\sigma_n )_n \subset (s_n )_n $ and a function $g\in BS^{p} (\mathbb{R},X)  $ such that 
\begin{eqnarray}
\lim_{n} \left( \int_{\left[ t,t+1\right] }\| f(s+\sigma_n )-g(s) \|^{p}ds\right)^{\frac{1}{p}}   =0, \label{charac APSp with sequences}
\end{eqnarray}
uniformly in $t\in \mathbb{R}$.
\end{theorem}

A sufficient condition for a Stepanov almost periodic function to be Bohr almost periodic is given in the next.

\begin{proposition}\label{PropCovApS}
Let  $f\in L^{p}_{loc} (\mathbb{R},X) $ for $ 1\leq p< \infty $. If $f$ is $S^p$-almost periodic and uniformly continuous, then $f$ is almost periodic. 
\end{proposition}
\begin{definition}
Let $ 1\leq p< \infty $. A  function $ f: \mathbb{R}\times X\longrightarrow Y$ such that $f(\cdot,x)\in L^{p}_{loc} (\mathbb{R},Y) $  for each $ x\in X $ is said to be $ S^{p} $-almost periodic in $ t $ uniformly with respect to $ x $ in $X$ if for each compact set $K$ in $X$, for all $\varepsilon > 0 $  there exists $l_{\varepsilon,K} > 0 $, such that for every $ a \in \mathbb{R}$ there exists $ \tau \in \left[a,a+l_{\varepsilon,K} \right]  $ satisfying: $$ \displaystyle \sup_{x\in K}\left( \int_{\left[ t,t+1\right] }\| f(s+\tau,x)-f(s,x) \|^{p}ds\right)^{\frac{1}{p}}  < \varepsilon \quad \mbox{for all}\; t\in \mathbb{R}.$$ 
The space of all such functions is denoted by $ APS^{p}U(\mathbb{R}\times X,Y) .$ 
\end{definition}
\subsection*{Hypotheses} Here, we list our main hypotheses:\\
\textbf{(H1)} The operators  $A(t) $, $ t \in \mathbb{R} $ satisfy the assumptions \eqref{AquTerCon1} and \eqref{AquTerCon2}. \\
\textbf{(H2)} The evolution family $(U(t,s))_{t\geq s}$ generated by $A(t) $, $ t \in \mathbb{R} $ has an exponential dichotomy on $\mathbb{R}$ with constants $N, \delta > 0$, projections $P(t), t \in \mathbb{R}$, and Green’s function $G$.\\
\textbf{(H3)} For each $x \in X$, $G(\cdot,\cdot)x $ is bi-almost periodic. \\
\textbf{(H4)} There exist $0 \leq \alpha < \beta < 1$ such that $X_{\alpha}^{t} = X_{\alpha}$ and $X_{\beta}^{t} = X_{\beta}$ for every $t\in \mathbb{R}$ with uniformly equivalent norms.\\
\textbf{(H5)} $f$ is locally Lipschitzian with respect to the second argument i.e., for all $\rho> 0$ there exists a nonegative scalar function $L_\rho (\cdot) $ such that 
$$ \| f(t,x)-f(t,y)\| \leq L_\rho (t)\|x-y \|_{\alpha}, \quad x,y\in B(0,\rho), \, t\in \mathbb{R} .$$

\begin{remark}
\textbf{(a)} In the case of an exponentially stable evolution family $(U(t,s))_{t\geq s}$ i.e., the case of Remark \ref{Remark1 Examples Gfun}-\textbf{(c)}. Then, $G(t,s)=U(t,s)$ for $ t\geq s $. Hence, hypothesis \textbf{(H3)} is reduced to the bi-almost periodicity of $(U(t,s))_{t\geq s}$ and this is hold foe example if $ A(\cdot) $ is $p$-periodic. \\ 
\textbf{(b)} In \cite{Man-Sch}, the authors proved that if $R(\omega,A(\cdot))$ is almost periodic for some $\omega \in \mathbb{R}$. Then, the associated Green function is bi-almost periodic. Here, in this paper, we improve this assumption by proving that \textbf{(H3)} holds, if just $ A(\cdot) $ is $ S^{1} $-almost periodic, see Section \ref{Section5}.
 \end{remark}

\section[New composition results]{New composition results of Stepanov almost periodic functions}\label{Section3}
In this section, we prove  a new composition result of $ S^{p} $-almost periodic functions (for $1\leq p<\infty$). We begin by the following useful characterization. 
\begin{lemma}\label{LemApSpCom1Chapter1}
Let $ 1\leq p < +\infty $ and $ f:\mathbb{R}\times Y\longrightarrow X $ be a function such that $ f(\cdot, x) \in L^{p}_{loc}(\mathbb{R}, X) $ for each $ x\in X .$ Then, $ f\in APS^{p}U(\mathbb{R}\times Y,X) $ if and only if the following hold:
\begin{itemize}
\item[\textbf{(i)}] For each $ x\in Y $, $ f(\cdot, x) \in APS^{p}(\mathbb{R},X) .$
\item[\textbf{(ii)}] $ f $ is $ S^{p} $-uniformly continuous with respect to the second argument on each compact subset $K$ in $Y$ in the following sense: for all $ \varepsilon>0 $ there exists $ \delta_{K, \varepsilon} $ such that  
for all $ x_{1}, x_{2} \in K $ one has 
\begin{eqnarray}
\| x_{1}-x_{2}\| \leq \delta_{K, \varepsilon} \Longrightarrow \left( \int_{t}^{t+1}\|f(s,x_{1})-f(s,x_{2}) \|^{p}ds\right) ^{\frac{1}{p}} \leq \varepsilon \quad \text{for all} \; t\in \mathbb{R}  . \label{ComCharSpChapter1}
\end{eqnarray}
\end{itemize}
\end{lemma}
%
\begin{proof}
Let $ f\in APS^{p}U(\mathbb{R}\times X,Y) $ and $ f^{b}:\mathbb{R}\times Y\longrightarrow L^{p}([0,1],X)  $ be the Bochner transform associated to $f$. It follows in view of  \cite[Lemma 2.6]{Cieutat},  that \textbf{(i)} is clearly satisfied and for each compact subset $K$ in $X,$ for all $ \varepsilon>0 $ there exists $ \delta_{K, \varepsilon} $ such that  
for all $ x_{1}, x_{2} \in K $ one has 
\begin{eqnarray*}
\| x_{1}-x_{2}\| \leq \delta_{K, \varepsilon} &\Longrightarrow & \| f^{b}(t,x_{1})-f^{b}(t,x_{2}) \| \leq \varepsilon \quad \text{for all} \; t\in \mathbb{R} .
\end{eqnarray*}
Since,
\begin{eqnarray*}
  \| f^{b}(t,x_{1})-f^{b}(t,x_{2}) \| &= &
 \left( \int_{[0,1]}\|(f^{b}(t,x_{1}))(s)-(f^{b}(t,x_{2}))(s) \|^{p}ds\right) ^{\frac{1}{p}}  \\ &=& 
 \left( \int_{t}^{t+1}\|f(s,x_{1})-f(s,x_{2}) \|^{p}ds\right) ^{\frac{1}{p}} \quad \text{for all} \; t\in \mathbb{R}.\\ 
\end{eqnarray*}
It follows that \eqref{ComCharSpChapter1} holds and then \textbf{(ii)} is achieved.\\ 
Conversely, let $ f:\mathbb{R}\times X\longrightarrow Y $ be a function such that $ f(\cdot, x) \in L^{p}_{loc}(\mathbb{R}, X) $ for each $ x\in Y . $ Assume that $ f $ satisfies \textbf{(i)-(ii)}. Let us fix a compact subset $ K $ in $Y$ and $ \varepsilon>0 .$ Since $ K $ is compact, it follows that there exists a finite subset $ \lbrace x_{1},...,x_{n} \rbrace  \subset K$ ($ n\in \mathbb{N}^{*} $) such that $ \displaystyle{K\subseteq \bigcup_{i=1}^{n} B(x_{i}, \delta_{K, \varepsilon} )} .$ Therefore, for $ x\in K $, there exist $ i=1,...,n $ satisfying $ \|x-x_{i}\|\leq \delta_{K, \varepsilon} $. Let $\tau \in \mathbb{R}$. Then, we obtain that
\begin{eqnarray}
 & &\left( \int_{t}^{t+1}\| f(s+\tau  ,x)-f(s,x) \|^{p}ds\right) ^{\frac{1}{p}} \leq   \left( \int_{t}^{t+1}\| f(s+\tau  ,x)-f(s+\tau , x_{i}) \|^{p}ds\right) ^{\frac{1}{p}} \nonumber \\ & & + \left( \int_{t}^{t+1}\| f(s+\tau  ,x_{i})-f(s,x_{i}) \|_{Y}^{p}ds\right) ^{\frac{1}{p}} \nonumber \\ & & +\left( \int_{t}^{t+1}\| f(s ,x_{i})-f(s,x) \|^{p}ds\right) ^{\frac{1}{p}},\; t\in \mathbb{R}. 
 \label{ComResFor1Chapter1}
\end{eqnarray}
Using \textbf{(i)}, we have for each $i=1,\dots,n $, $f(\cdot,x_{i}) \in AAS^{p}(\mathbb{R},X) $. Hence, there exists $l_{K, \varepsilon}>0$ such that for all $ a\in \mathbb{R} $ there exists $ \tau \in [a,a+l_{K, \varepsilon}] $ satisfying
\begin{eqnarray}
 \left( \int_{t}^{t+1}\| f(s+\tau ,x_{i})-f(s,x_{i}) \|^{p}ds\right) ^{\frac{1}{p}} \leq \frac{\varepsilon}{3} \quad \text{for all} \; t\in \mathbb{R}. \label{ComResFor2Chapter1}
\end{eqnarray}  
Since $\| x-x_{i}\| \leq \delta_{K,\delta} $ and  by \textbf{(ii)}, we claim that 
\begin{eqnarray}
 \left( \int_{t}^{t+1}\| f(s+\tau  ,x)-f(s+\tau , x_{i}) \|^{p}ds\right) ^{\frac{1}{p}} \leq \frac{\varepsilon}{3} \label{ComResFor3Chapter1} \quad \text{for all} \; t\in \mathbb{R},
\end{eqnarray}
and 
\begin{eqnarray}
 \left( \int_{t}^{t+1}\| f(s ,x)-f(s, x_{i}) \|^{p}ds\right) ^{\frac{1}{p}} \leq \frac{\varepsilon}{3} \quad \text{for all} \; t\in \mathbb{R}. \label{ComResFor4Chapter1}
\end{eqnarray}
Consequently, we replace \eqref{ComResFor2Chapter1}, \eqref{ComResFor3Chapter1} and \eqref{ComResFor4Chapter1} in \eqref{ComResFor1Chapter1},  we obtain that
\begin{eqnarray*}
\sup_{x\in K}\left( \int_{t}^{t+1}\| f(s+\tau ,x)-f(s,x) \|^{p}ds\right) ^{\frac{1}{p}} \leq \varepsilon  \quad \text{for all} \; t\in \mathbb{R} . 
\end{eqnarray*}
\end{proof}
By Lemma \ref{LemApSpCom1Chapter1}, we deduce the following composition result.
\begin{theorem}\label{ThmComApSpChapter1}
Let $ 1\leq p < +\infty $ and $ f\in APS^{p}U(\mathbb{R}\times Y,X) $. Assume that $x \in AP (\mathbb{R},Y)$. Then, $ f(\cdot,x(\cdot)) \in APS^{p}(\mathbb{R}, X).$ 
\end{theorem}
\begin{proof}
Let $t, \tau \in \mathbb{R}$. Then, we have
\begin{eqnarray*}
\hspace{-1cm}& &\left( \int_{t}^{t+1}\| f(s+\tau ,x(s+\tau))-f(s,x(s)) \|^{p}ds\right) ^{\frac{1}{p}} \nonumber \\ & \leq & \left( \int_{t}^{t+1}\| f(s+\tau ,x(s+\tau))-f(s+\tau , x(s)) \|^{p}ds\right)   ^{\frac{1}{p}} \\ & & + \left( \int_{t}^{t+1}\| f(s+\tau ,x(s))-f(s, x(s)) \|^{p}ds\right)^{\frac{1}{p}}.   
\end{eqnarray*}
Moreover, given $ K:=\overline{\lbrace x(t):\ t\in \mathbb{R} \rbrace} $ a compact subset of $Y$ and $ \varepsilon >0 $ be fixed. Using Lemma \ref{LemApSpCom1Chapter1}\textbf{-(ii)} it follows that there exists $ \delta_{\varepsilon, K} > 0$ such that \eqref{ComCharSpChapter1} holds. 
Let $ \varepsilon>0 $, since $ u \in AP(\mathbb{R},Y) $, it follows that, there exists $ l_{\varepsilon}>0 $ such that  every interval of length $  l_{\varepsilon} $ contains an element $ \tau $ such that  $$\| x(s+\tau )- x(s)\| \leq \delta_{\varepsilon, K} \; \text{ for all } s\in \mathbb{R}.$$  Moreover, for each $ s\in \mathbb{R}  $, we have $ x(s) \in K $. Hence, 
\begin{eqnarray}
 \left( \int_{t}^{t+1}\| f(s+\tau ,x(s+\tau))-f(s+\tau , x(s)) \|^{p}ds\right)   ^{\frac{1}{p}}  \leq \frac{\varepsilon}{4 } \label{ComAPSpFor1Chapter1}
 \end{eqnarray}
Furthermore, since $ K $ is compact, it follows that, there exists a finite subset $ \lbrace x_{1},...,x_{n} \rbrace  \subset K$ ($ n\in \mathbb{N}^{*} $) such that $ \displaystyle{K\subseteq \bigcup_{i=1}^{n} B(x_{i}, \delta_{K, \varepsilon} )}.$ Then, for all $ t\in \mathbb{R} $ there exists $ i(t)=1,...,n $ such that $ \|x(t)-x_{i(t)} \| \leq \delta_{K,\varepsilon} .$ Thus
\begin{eqnarray}
 \hspace{-0.5cm}\left( \int_{t}^{t+1}\| f(s+\tau ,x(s))-f(s+\tau , x_{i(t)})) \|^{p}ds\right)^{\frac{1}{p}} \leq \frac{\varepsilon}{4}, \label{ComAPSpFor2Chapter1} 
\end{eqnarray}
and 
\begin{eqnarray}
 \left( \int_{t}^{t+1}\| f(s ,x(s))-f(s, x_{i(t)})) \|^{p}ds\right)^{\frac{1}{p}} \leq \frac{\varepsilon}{4} . \label{ComAPSpFor3Chapter1}
\end{eqnarray}
Using Lemma \ref{LemApSpCom1Chapter1}\textbf{-(i)}, we get that
\begin{eqnarray}
 \left( \int_{t}^{t+1}\| f(s+\tau ,x_{i(t)})-f(s, x_{i(t)})) \|^{p}ds\right)^{\frac{1}{p}} \leq \frac{\varepsilon}{4} . \label{ComAPSpFor4Chapter1}
\end{eqnarray}
Consequently, by  \eqref{ComAPSpFor1Chapter1}, \eqref{ComAPSpFor2Chapter1}, \eqref{ComAPSpFor3Chapter1} and \eqref{ComAPSpFor4Chapter1}, we obtain that
\begin{eqnarray*}
\left( \int_{t}^{t+1}\| f(s+\tau ,x(s+\tau))-f(s,x(s)) \|^{p}ds\right) ^{\frac{1}{p}} \leq \frac{\varepsilon}{4 } + \frac{\varepsilon}{4} + \frac{\varepsilon}{4}+ \frac{\varepsilon}{4} = \varepsilon \quad \text{for all}\; t\in \mathbb{R}.
\end{eqnarray*}
This proves the result.
\end{proof}
\section{Almost periodic solutions to semilinear evolution equations}\label{Section4}
In this section, we prove the existence and uniqueness of almost periodic solutions for our semilinear evolution equation \eqref{Eq_1}. For that purpose, we set the following associated inhomogeneous linear equation:  
\begin{eqnarray}
  u'(t)= A(t)u(t) + h(t) \quad \text{for}\; t \in \mathbb{R} \label{EqLinHomo}
\end{eqnarray}
where $ h:\mathbb{R}\longrightarrow X $ is locally integrable. A mild solution of equation  \eqref{EqLinHomo} is the continuous function $ u: \mathbb{R} \longrightarrow X^{t}_{\alpha} $  which satisfies the following variation of constants formula:
\begin{eqnarray}
u(t)=U(t,\sigma)u(\sigma)+ \int_{\sigma}^{t} U(t,s)h(s)ds \qquad \mbox{for all} \; t\geq \sigma. \label{MildSolLinHom}
\end{eqnarray}
Similarly, we define mild solution of equation \eqref{Eq_1} by taking $f(\cdot, u(\cdot))=h(\cdot)$.\\

 Now, we prove the existence and uniqueness of almost periodic mild solutions to \eqref{EqLinHomo} under the assumption that $ h $ is just $ S^{p} $-almost periodic ($ 1 \leq p < \infty  $). For technical reasons, we distinguish two cases, $p=1$ and $1<p<\infty$.  
 
\begin{theorem}\label{ThmExiBC1}
 Let $ h \in L^{\infty}(\mathbb{R},X) $ and assume that \textbf{(H1)}-\textbf{(H2)} and \textbf{(H4)} are satisfied. Then, equation \eqref{EqLinHomo} has a unique bounded mild solution $ u: \mathbb{R}\longrightarrow X_{\alpha} $ given by
\begin{eqnarray}
u(t)=\int_{\mathbb{R}}G (t,s)h(s)ds  \quad \text{for all}\;  t\in  \mathbb{R}.\label{MildSolAp}
\end{eqnarray}
\end{theorem}
\begin{proof}
Let  $ h \in L^{\infty}(\mathbb{R},X) $.  First, we show that the integral formula in \eqref{MildSolAp} is well defined and yields a continuous bounded  function in $X_{\alpha}$.  Indeed, using the estimates \eqref{prop2.3 estim1} and \eqref{prop2.3 estim2}, one has
\begin{eqnarray*}
\| \int_{\mathbb{R}}G (t,s)h(s)ds \|_{\alpha} 
&=& \| \int_{-\infty}^{t}U (t,s)P(s)h(s)ds -\int^{+\infty}_{t} U (t,s)Q(s)h(s)ds \|_{\alpha} \\
&\leq & \int_{-\infty}^{t} \| U (t,s)P(s)h(s) \|_{\alpha} ds + \int^{+\infty}_{t} \| U (t,s)Q(s)h(s) \|_{\alpha} ds \\
&\leq & m(\alpha ) \int_{-\infty}^{t} (t-s)^{- \alpha} e^{- \gamma (t-s)}  \| h(s) \| ds  \\ &+&  c(\alpha) \int^{+\infty}_{t} e^{- \delta (s-t)}  \| h(s) \|  ds \\
&\leq  &\left(   m(\alpha )\gamma^{\alpha -1} \Gamma(1- \alpha ) +  c(\alpha)\delta^{-1}\right)    \| h \|_{\infty}, \; t\in \mathbb{R}.  
\end{eqnarray*}
Thus $ v:= \displaystyle{ \int_{\mathbb{R}}G (\cdot,s)h(s)ds }$ defines a bounded function from $\mathbb{R}$ to $ X_\alpha  $ and it is clear that it is also continuous. Now, we show that $v$ is the unique mild solution of equation \eqref{EqLinHomo}. Let $ \sigma \in \mathbb{R} $ be fixed and $ t\geq \sigma $. Then, by the property of evolution families, we claim that
\begin{eqnarray*}
v(t)&=& \int_{-\infty}^{t}U (t,s)P(s)h(s)ds -\int^{+\infty}_{t} \tilde{U} (t,s)Q(s)h(s)ds  \\
&=& \int_{-\infty}^{\sigma}U (t,s)P(s)h(s)ds+ \int_{\sigma}^{t}U (t,s)P(s)h(s)ds \\ 
&&-\int_{t}^{\sigma}  \tilde{U} (t,s)Q(s)h(s)ds -\int^{+\infty}_{\sigma} \tilde{U} (t,s)Q(s)h(s)ds \\
&=& U(t,\sigma)v(\sigma) +  \int_{\sigma}^{t}U (t,s)h(s)ds
\end{eqnarray*}
where $ P(t) $ and $Q(t)$, $t\in \mathbb{R}$ are the  corresponding projections of the exponential dichotomy.
Conversely, let $u$ be the bounded mild solution of equation \eqref{EqLinHomo} defined by 
$$ u(t)=U(t,\sigma)u(\sigma)+ \int_{\sigma}^{t} U(t,s)h(s)ds .$$
Then, 
$$ P(t)u(t)= U(t,\sigma)P(\sigma) u(\sigma) + \int_{\sigma}^{t} U(t,s)P(s)h(s)ds $$  and $$ Q(t)u(t)= U(t,\sigma)Q(\sigma) u(\sigma) + \int_{\sigma}^{t} \tilde{U}(t,s)Q(s)h(s)ds $$
Hence, by the boundedness of $ u $ and in view of the estimates \eqref{prop2.3 estim1} and \eqref{prop2.3 estim2}, it follows using the dominated convergence Theorem, by letting $ \sigma \rightarrow - \infty $ and $ \sigma \rightarrow - \infty $ respectively, that  
$$ P(t)u(t)= \int_{-\infty}^{t} U(t,s)P(s)h(s)ds $$  and $$ Q(t)u(t)=- \int_{t}^{+\infty} U(t,s)Q(s)h(s)ds $$
Since the decomposition $ u(t)=P(t)u(t)+Q(t)u(t)$ is unique, we obtain that $u$ is uniquely determined by the integral formula given by $v$. 
\end{proof}
For $1<p<\infty$, we dot not need to assume the boundedness of $ h $, just to be $ S^p $-bounded is sufficient. Indeed, we have the following main result. 
\begin{theorem}
Let $ h \in BS^{p}(\mathbb{R},X) $ and assume that \textbf{(H1)}-\textbf{(H2)} and \textbf{(H4)} hold. Then, equation \eqref{EqLinHomo} has a unique bounded mild solution $ u: \mathbb{R}\longrightarrow X_{\alpha} $ defined by the integral formula  \eqref{MildSolAp}.
\end{theorem}
\begin{proof}
The fact that the mild solution of equation \eqref{EqLinHomo} is defined uniquely by  \eqref{MildSolAp} is similar as in the proof of Theorem \ref{ThmExiBC1}. Hence, it suffices to prove that $u \in BC(\mathbb{R},X_{\alpha})$. Let $ h \in BS^{p}(\mathbb{R},X) $, using H\"older inequality, we obtain that 
\begin{eqnarray*}
&& \| \int_{\mathbb{R}}G (t,s)h(s)ds \|_{\alpha} \\
&=& \| \int_{-\infty}^{t}U (t,s)P(s)h(s)ds -\int^{+\infty}_{t} U (t,s)Q(s)h(s)ds \|_{\alpha} \\
&\leq & \int_{-\infty}^{t} \| U (t,s)P(s)h(s) \|_{\alpha} ds + \int^{+\infty}_{t} \| U (t,s)Q(s)h(s) \|_{\alpha} ds \\
&\leq & m(\alpha ) \int_{-\infty}^{t} (t-s)^{- \alpha} e^{- \gamma (t-s)}  \| h(s) \| ds  +  c(\alpha) \int^{+\infty}_{t} e^{- \delta (s-t)}  \| h(s) \|  ds \\
&\leq & m(\alpha ) \left(  \int_{-\infty}^{t} (t-s)^{-q\alpha} e^{-q\frac{\gamma}{2} (t-s)}  ds \right)^{\frac{1}{q}} \left( \int_{-\infty}^{t} e^{-p\frac{\gamma}{2} (t-s)} \| h(s)\|^{p} ds  \right)^{\frac{1}{p}}   \\ 
&+&  c(\alpha) \left( \int^{+\infty}_{t} e^{-q\frac{\delta}{2} (s-t)}  ds\right)^{\frac{1}{q}} \left( \int^{+\infty}_{t} e^{- p\frac{\delta}{2} (s-t)}  \| h(s) \|^{p}  ds\right)^{\frac{1}{p}}  \\
&= & m(\alpha )\left(  \dfrac{2}{q\gamma}\right)^{\alpha}  \Gamma(q(1-\alpha))^{\frac{1}{q}} \left( \sum_{k\geq 1} \int_{t-k}^{t-k+1} e^{-p\frac{\gamma}{2} (t-s)} \| h(s)\|^{p} ds  \right)^{\frac{1}{p}}   \\
 &+&  c(\alpha) \left( \dfrac{2}{q\delta} \right)^{\frac{1}{q}} \left(\sum_{k\geq 1} \int_{t+k-1}^{t+k}e^{- p\frac{\delta}{2} (s-t)}  \| h(s) \|^{p}  ds\right)^{\frac{1}{p}}  \\
&\leq & m(\alpha )\left(  \dfrac{2}{q\gamma}\right)^{\alpha}  \Gamma(q(1-\alpha))^{\frac{1}{q}} \left( \sum_{k\geq 1}e^{-p\frac{\gamma}{2} (k-1)} \int_{t-k}^{t-k+1}  \| h(s)\|^{p} ds  \right)^{\frac{1}{p}}  \\
 &+&  c(\alpha) \left( \dfrac{2}{q\delta} \right)^{\frac{1}{q}} \left(\sum_{k\geq 1}e^{- p\frac{\delta}{2} (k-1)}  \int_{t+k-1}^{t+k} \| h(s) \|^{p}  ds\right)^{\frac{1}{p}} \\
 &=&  \left[  m(\alpha )\left(  \dfrac{2}{q\gamma}\right)^{\alpha}  \Gamma(q(1-\alpha))^{\frac{1}{q}} \dfrac{e^{\frac{\gamma}{2}}}{e^{\frac{\gamma}{2}}-1}+ c(\alpha) \left( \dfrac{2}{q\delta} \right)^{\frac{1}{q}}\dfrac{e^{\frac{\delta}{2}}}{e^{\frac{\delta}{2}}-1} \right] \| h\|_{BS^{p}}, \ t\in \mathbb{R}.
\end{eqnarray*}
This proves the result.
\end{proof}
Next result, we show that the unique bounded mild solution $u$ is almost periodic.
\begin{theorem}\label{ThmExiAP1}
 Let $ h \in APS^{1}(\mathbb{R},X) \cap L^{\infty}(\mathbb{R},X) $ and assume that \textbf{(H1)}-\textbf{(H4)} are satisfied. Then, equation \eqref{EqLinHomo} has a unique almost periodic mild solution $ u: \mathbb{R}\longrightarrow X_{\alpha} $ given by the integral formula  \eqref{MildSolAp}.
\end{theorem}
\begin{proof}
Let $ h \in APS^{1}(\mathbb{R},X) \cap L^{\infty}(\mathbb{R},X) $ and $u$ be the unique bounded mild solution of equation \eqref{EqLinHomo} provided by Theorem \ref{ThmExiBC1}. Then, we have
\begin{eqnarray*}
u(t) &=&\int_{-\infty}^{t}U (t,s)P(s)h(s)ds -\int^{+\infty}_{t} \tilde{U} (t,s)Q(s)h(s)ds \\
&:=& u^{P}(t)+ u^{Q}(t), \quad  t\in \mathbb{R}
\end{eqnarray*}
where $ \displaystyle  u^{P}(t)= \int_{-\infty}^{t}U (t,s)P(s)h(s)ds $ and $ \displaystyle u^{Q}(t)= -\int^{+\infty}_{t} \tilde{U} (t,s)Q(s)h(s)ds $.\\
Let $ t, \ \tau \in \mathbb{R}$. Thus, in view of Cauchy–Schwarz inequality, we have
\begin{eqnarray*}
&&\|  u^{P}(t+\tau )- u^{P}(t) \|_{\alpha} \\
&=& \| \int_{-\infty}^{t+\tau}U (t+\tau,s)P(s)h(s)ds -\int_{-\infty}^{t} U (t,s)P(s)h(s)ds \|_{\alpha} \\
&\leq & \int_{-\infty}^{t} \| U (t,s )P(s)\left[ h(s+\tau)-h(s) \right] \|_{\alpha} ds \\
 &+& \int_{-\infty}^{t} \| \left[ U (t+\tau, s+\tau)P(s+\tau)-U (t, s)P(s)\right]h(s+\tau )  \|_{\alpha} ds \\
&\leq & m(\alpha ) \sqrt{2\gamma^{-1} \| h\|_{\infty} \Gamma(2\alpha+1) }  \left( \sum_{k\geq 1} e^{-\gamma (k-1)}\int_{t-k}^{t-k+1}  \| h(s+\tau)-h(s)\| ds  \right)^{\frac{1}{2}}\\
 &+& m(\alpha )  \sum_{k\geq 1}\left(  \int_{t-k}^{t-k+1} \| \left[ U (t+\tau, s+\tau)P(s+\tau)-U (t, s)P(s)\right]h(s+\tau )  \|_{\alpha} ds\right)^{\frac{1}{2}} \\ &\times & e^{- \frac{\gamma}{4}(k-1)}  \left( 2 \int_{t-k}^{t-k+1} (t-s)^{- \alpha} e^{- \frac{\gamma}{2}( (t-s)}  \| h(s+\tau )  \| ds  \right)^{\frac{1}{2}} \\
 &\leq & m(\alpha ) \sqrt{2\gamma^{-1} \| h\|_{\infty} \Gamma(2\alpha+1) }  \left( \sum_{k\geq 1} e^{-\gamma (k-1)}\int_{t-k}^{t-k+1}  \| h(s+\tau)-h(s)\| ds  \right)^{\frac{1}{2}}\\
 &+& m(\alpha ) \sqrt{2 \left( \frac{2}{\gamma} \right)^{1-\alpha}\Gamma(1+\alpha) \|h\|_{\infty} }  \\ &\times &  \sum_{k\geq 1}e^{- \frac{\gamma}{4}(k-1)}  \left(  \int_{t-k}^{t-k+1} \| \left[ U (t+\tau, s+\tau)P(s+\tau)-U (t, s)P(s)\right]h(s+\tau )  \|_{\alpha} ds\right)^{\frac{1}{2}} \\
\end{eqnarray*}
Let $\varepsilon>0$, since $ h \in APS^{1}(\mathbb{R},X)  $, it follows that, there exists $ l_\varepsilon >0 $ such that every interval of length $ l_\varepsilon >0 $ contains an element $\tau$ such that 
\begin{eqnarray}
 \int_{t}^{t+1}  \| h(s+\tau)-h(s)\| ds \leq \dfrac{\varepsilon^{2}\sqrt{e^{\gamma}-1}}{4 m(\alpha ) \sqrt{2\gamma^{-1} \| h\|_{\infty} \Gamma(2\alpha+1) e^{\gamma} }} \;\;   \text{uniformly for } \, t\in \mathbb{R}. \label{Almost perio h func}
\end{eqnarray}
Furthermore, by hypothesis \textbf{(H3)}, we can find the same almost period $\tau$ such that for each $x \in X$, we have
\begin{eqnarray}
\| U (t+\tau, s+\tau)P(s+\tau)x-U (t, s)P(s)x \|_{\alpha}  \leq  \dfrac{\varepsilon^{2}(e^{\frac{\gamma}{4}}-1)}{4 m(\alpha ) \|h\|_{\infty}\sqrt{2 \left( \frac{2}{\gamma} \right)^{1-\alpha}\Gamma(1+\alpha)  } } \| x\|  \label{Almost perio Green func}
\end{eqnarray}
uniformly for $ t,s\in \mathbb{R}$. Therefore,
$$ \|  u^{P}(t+\tau )- u^{P}(t) \|_{\alpha} \leq \dfrac{\varepsilon}{4}  +  \dfrac{\varepsilon}{4} = \dfrac{\varepsilon}{2} \;\;   \text{uniformly for } \, t\in \mathbb{R}.$$
By the same way, we prove that 
$$ \|  u^{Q}(t+\tau )- u^{Q}(t) \|_{\alpha} \leq  \dfrac{\varepsilon}{2} \;\;   \text{uniformly for } \, t\in \mathbb{R}.$$
Hence, $$ \|  u(t+\tau )- u(t) \|_{\alpha} \leq  \varepsilon \;\;   \text{uniformly for } \, t\in \mathbb{R}.$$
This proves the almost periodicity of the mild solution $u$.
\end{proof}
Similarly, for $1<p<\infty$, we can prove easily the following existence result of a unique almost periodic solution for equation \eqref{EqLinHomo}.
\begin{theorem}\label{ThmExiAP2}
Let $ h \in APS^{p}(\mathbb{R},X) $ and assume that \textbf{(H1)}-\textbf{(H4)} are satisfied. Then, equation \eqref{EqLinHomo} has a unique almost periodic mild solution $ u: \mathbb{R}\longrightarrow X_{\alpha} $ given by the integral formula  \eqref{MildSolAp}.
\end{theorem}
Now, we return to equation \eqref{Eq_1} and we give our main results on the existence and uniqueness of almost periodic solutions.  Without loss of generalities, we assume that $ f(t,0)=0,\; t\in \mathbb{R} $. If it is not the case, we can take $\tilde{f}(t,x)=f(t,x)-f(t,0)$ instead of $f$. Furthermore, consider $u\in AP(\mathbb{R},X_\alpha)$, then using Theorem \ref{ThmComApSpChapter1}, the function $h=f(\cdot,u(\cdot)) \in APS^{p}(\mathbb{R},X)$ for all $1\leq p<\infty$. Moreover, let $ f(\cdot,x) \in APS^{1}(\mathbb{R},X ) \cap L^{\infty}(\mathbb{R},X) $ (resp. $f(\cdot,x) \in APS^{p}(\mathbb{R},X )$, $ 1<p<\infty $)  for each $x \in X_{\alpha}$. Then, in view of Theorems \ref{ThmExiAP1} and \ref{ThmExiAP2},
 the map $F: AP(\mathbb{R},X_\alpha ) \longrightarrow AP(\mathbb{R},X_\alpha )$ given  by $$   Fu(t)=  \int_{\mathbb{R}}G (t,s)f(s,u(s))ds, \quad t\in \mathbb{R} $$ is well-defined.
\begin{theorem}\label{Main Theorem1 Eq1}
Let $ f(\cdot,x) \in APS^{1}(\mathbb{R},X ) \cap L^{\infty}(\mathbb{R},X) $ (resp. $f(\cdot,x) \in APS^{p}(\mathbb{R},X )$, $ 1<p<\infty $)  for each $x \in X_{\alpha}$ satisfying \textbf{(H5)} with $ L_{\rho} \in BS^{p}(\mathbb{R},\mathbb{R}^{+})$ for $ 1<p<\infty $.  
Assume that \textbf{(H1)}-\textbf{(H4)} hold. If there exists $\rho >0$ such that $$ \| L_\rho \|_{BS^{p}} <   \left[  m(\alpha )\left(  \dfrac{2}{q\gamma}\right)^{\alpha}  \Gamma(q(1-\alpha))^{\frac{1}{q}} \dfrac{e^{\frac{\gamma}{2}}}{e^{\frac{\gamma}{2}}-1}+ c(\alpha) \left( \dfrac{2}{q\delta} \right)^{\frac{1}{q}}\dfrac{e^{\frac{\delta}{2}}}{e^{\frac{\delta}{2}}-1} \right]^{-1} .$$ Then, equation \eqref{Eq_1} has a unique almost periodic solution $u: \mathbb{R}\longrightarrow X_{\alpha}$  with $ \| u\|_{\infty} \leq \rho $.
\end{theorem}
\begin{proof}
Consider the set $\Lambda_{\rho}^{AP} := \{ v \in AP(\mathbb{R},X_\alpha ) : \, \|v\|_{\infty} \leq \rho  \}$
and define the map $F: \Lambda_{\rho}^{AP}  \longrightarrow AP(\mathbb{R},X_\alpha )$  by 
$$   Fu(t)=  \int_{\mathbb{R}}G (t,s)f(s,u(s))ds, \quad t\in \mathbb{R}. $$
First, we show that $ F \Lambda_{\rho}^{AP} \subset \Lambda_{\rho}^{AP}$. Indeed, let $u \in \Lambda_{\rho}^{AP}$. Then, by assumptions on $f$, we obtain that
\begin{eqnarray*}
\|Fu(t)\|_{\alpha} &\leq & \int_{\mathbb{R}}\| G (t,s)f(s,u(s)) \|_{\alpha} ds \\ 
&\leq &  m(\alpha ) \int_{-\infty}^{t} (t-s)^{- \alpha} e^{- \gamma (t-s)}  \| f(s,u(s))-f(s,0) \| ds \\ &+& c(\alpha) \int^{+\infty}_{t} e^{- \delta (s-t)}  \|f(s,u(s))-f(s,0) \|  ds \\
 &\leq & \rho m(\alpha ) \int_{-\infty}^{t} (t-s)^{- \alpha} e^{- \gamma (t-s)} L_{\rho}(s) ds \\
 &+& \rho c(\alpha) \int^{+\infty}_{t} e^{- \delta (s-t)} L_{\rho}(s)   ds \\
 &\leq & \rho m(\alpha ) \left( \int_{-\infty}^{t} (t-s)^{- q\alpha} e^{- q\frac{ \gamma}{2} (t-s)} ds\right)^{\frac{1}{q}} \left(  \int_{-\infty}^{t} e^{-p \frac{ \gamma}{2} (t-s)} |L_{\rho}(s) |^{p} ds\right)^{\frac{1}{p}} \\
 &+& \rho c(\alpha) \left( \int^{+\infty}_{t} e^{- q\frac{ \delta}{2} (s-t)} ds\right)^{\frac{1}{q}}  \left(  \int^{+\infty}_{t} e^{-p \frac{ \delta}{2} (s-t)} | L_{\rho}(s) |^{p} ds \right)^{\frac{1}{p}}  \\
& \leq & \rho  \left[  m(\alpha )\left(  \dfrac{2}{q\gamma}\right)^{\alpha}  \Gamma(q(1-\alpha))^{\frac{1}{q}} \dfrac{e^{\frac{\gamma}{2}}}{e^{\frac{\gamma}{2}}-1}+ c(\alpha) \left( \dfrac{2}{q\delta} \right)^{\frac{1}{q}}\dfrac{e^{\frac{\delta}{2}}}{e^{\frac{\delta}{2}}-1} \right] \| L_{\rho}\|_{BS^{p}}
\\ &\leq & \rho, \quad t\in \mathbb{R}.
\end{eqnarray*}
Hence, $ F \Lambda_{\rho}^{AP} \subset \Lambda_{\rho}^{AP}$. Therefore, let $u, \ v \in  \Lambda_{\rho}^{AP}$. Then, a straightforward calculation yields
\begin{eqnarray*}
\|Fu(t)-Fv(t)\|_{\alpha} &\leq & \int_{\mathbb{R}}\| G (t,s)\left[ f(s,u(s))-f(s,u(s))\right]  \|_{\alpha} ds \\ 
&\leq &  m(\alpha ) \int_{-\infty}^{t} (t-s)^{- \alpha} e^{- \gamma (t-s)}  \| f(s,u(s))-f(s,v(s)) \| ds \\ &+& c(\alpha) \int^{+\infty}_{t} e^{- \delta (s-t)}  \|f(s,u(s))-f(s,v(s)) \|  ds \\
 &\leq &  m(\alpha ) \int_{-\infty}^{t} (t-s)^{- \alpha} e^{- \gamma (t-s)} L_{\rho}(s) ds  \| u-v\|_{\infty} \\
 &+&  c(\alpha) \int^{+\infty}_{t} e^{- \delta (s-t)} L_{\rho}(s)   ds   \| u-v\|_{\infty} \\
&\leq & \| L_{\rho}\|_{BS^{p}}  \left[  m(\alpha )\left(  \dfrac{2}{q\gamma}\right)^{\alpha}  \Gamma(q(1-\alpha))^{\frac{1}{q}} \dfrac{e^{\frac{\gamma}{2}}}{e^{\frac{\gamma}{2}}-1}+ c(\alpha) \left( \dfrac{2}{q\delta} \right)^{\frac{1}{q}}\dfrac{e^{\frac{\delta}{2}}}{e^{\frac{\delta}{2}}-1} \right]   \| u-v\|_{\infty}, 
\end{eqnarray*}
for all $ t \in \mathbb{R} $. Thus, by assumption, the map $F$ defines a contraction in  $\Lambda_{\rho}^{AP}$. Consequently, by Banach fixed point Theorem, we obtain the existence and uniqueness of a solution $ u \in  \Lambda_{\rho}^{AP}$. This proves the result.
\end{proof}
\begin{theorem}\label{Main Theorem2 Eq1}
Let $ f(\cdot,x) \in APS^{1}(\mathbb{R},X ) \cap L^{\infty}(\mathbb{R},X) $ (resp. $f(\cdot,x) \in APS^{p}(\mathbb{R},X )$, $ 1<p<\infty $)  for each $x \in X_{\alpha}$ satisfying \textbf{(H5)} with $ L_{\rho} \in L^{\infty}(\mathbb{R},\mathbb{R}^{+})$.  
Assume that \textbf{(H1)}-\textbf{(H4)} hold. If there exists $\rho >0$ such that 
\begin{eqnarray}
\| L_\rho \|_{\infty} < \left( m(\alpha)\gamma^{\alpha -1} \Gamma(1+\alpha) +c(\alpha)\delta^{-1}\right)^{-1} . \label{Contra cond Theorem2}
\end{eqnarray}
 Then, equation \eqref{Eq_1} has a unique almost periodic solution $u: \mathbb{R}\longrightarrow X_{\alpha}$  with $ \| u\|_{\infty} \leq \rho $.
\end{theorem}
\begin{proof}
Consider the map $F: \Lambda_{\rho}^{AP}  \longrightarrow AP(\mathbb{R},X_\alpha )$  by 
$$   Fu(t)=  \int_{\mathbb{R}}G (t,s)f(s,u(s))ds, \quad t\in \mathbb{R}. $$
We argue as the same as in the proof of Theorem \ref{Main Theorem1 Eq1}. Indeed, let $u \in \Lambda_{\rho}^{AP}$. Then, by assumptions on $f$, we claim that
\begin{eqnarray*}
\|Fu(t)\|_{\alpha} &\leq & \int_{\mathbb{R}}\| G (t,s)f(s,u(s)) \|_{\alpha} ds \\ 
&\leq &  m(\alpha ) \int_{-\infty}^{t} (t-s)^{- \alpha} e^{- \gamma (t-s)}  \| f(s,u(s))-f(s,0) \| ds \\ &+& c(\alpha) \int^{+\infty}_{t} e^{- \delta (s-t)}  \|f(s,u(s))-f(s,0) \|  ds \\
&\leq &\rho  \left[ m(\alpha ) \gamma^{\alpha -1} \Gamma(1+\alpha)  + c(\alpha) \delta^{-1} \right] \| L_\rho \|_{\infty} \\ &\leq & \rho .
\end{eqnarray*}
Hence, $ F \Lambda_{\rho}^{AP} \subset \Lambda_{\rho}^{AP}$. Let $u, \ v \in  \Lambda_{\rho}^{AP}$. Then, using the same calculus yields,
\begin{eqnarray*}
\|Fu(t)-Fv(t)\|_{\alpha} &\leq & \int_{\mathbb{R}}\| G (t,s)\left[ f(s,u(s))-f(s,u(s))\right]  \|_{\alpha} ds \\ 
&\leq &  m(\alpha ) \int_{-\infty}^{t} (t-s)^{- \alpha} e^{- \gamma (t-s)}  \| f(s,u(s))-f(s,v(s)) \| ds \\ &+& c(\alpha) \int^{+\infty}_{t} e^{- \delta (s-t)}  \|f(s,u(s))-f(s,v(s)) \|  ds \\
&\leq & \| L_\rho \|_{\infty}  \left[ m(\alpha )\gamma^{\alpha -1} \Gamma(1+\alpha)  + c(\alpha) \delta^{-1} \right]  \| u-v\|_{\infty}, \quad t\in \mathbb{R}.
\end{eqnarray*}
Thus, $F$ defines a contraction map in  $\Lambda_{\rho}^{AP}$.  Hence, we conclude by Banach fixed point Theorem that there exists a unique solution $ u \in  \Lambda_{\rho}^{AP}$. This proves the result.
\end{proof}
\section{Application}\label{Section5}
Consider the following ecological two--species competition model with diffusion and time--dependent parameters i.e., a nonautonomous system with two reaction-diffusion equations of two interacting density species $ u(t,x)$ and $ v(t,x) $ at location $x$ that belongs to a habitat $\Omega \subset \R^N$ ($N \geq 1$), in a generalized almost periodic environment, namely
\begin{equation*}
 \left\{
    \begin{aligned}
     \frac{\partial}{\partial t}u(t,x)&=  d_1 (t)\Delta u(t,x) + a(t)u(t, x)- c_1 (t)\dfrac{ v(t,x) u(t,x) }{1+| \nabla v(t,x)|}, \, t \in \mathbb{R}, x\in \Omega, \\   
     \frac{\partial}{\partial t}u(t,x)&=  d_2 (t)\Delta v(t,x) - b(t)v(t, x)+ c_2 (t)\dfrac{ u(t,x) v(t,x)}{1+|\nabla u(t,x)|}, \,t \in \mathbb{R}, x\in \Omega, \\
     u(t,x)|_{\partial \Omega}& = 0 ; \  v(t,x)|_{\partial \Omega} =0, \; t \in \mathbb{R}, x\in \partial\Omega  
    \end{aligned}
  \right. 
\end{equation*}
where, \\
$\bullet$ $\Omega \subset \R^N$ ($N \geq 1$) is an open bounded subset with sufficiently smooth boundary i.e.,  with Lipschitz boundary conditions.\\
$ \bullet $ $ \Delta :=\displaystyle \sum_{k=1}^{n}\dfrac{\partial^{2}}{\partial \xi_{i}^{2}}$ is the Laplace operator on $\Omega  $, $ d_{i} \in C^{\mu}(\mathbb{R}) $, (eventually here $\mu=1$),  $i=1,2$ are the diffusion terms corresponding to the populations $  u(t,x) $ and $v(t,x)$ respectively, such that $0< \inf_{t \in \mathbb{R}}(d_{i}) \leq \sup _{t\in \mathbb{R}}(d_{i}) <\infty.$\\
$ \bullet $ $ a, \ b \in L^{1}_{loc}(\mathbb{R}, \mathbb{R}^{+})  $ correspond to the growth terms in the absence  of interaction, of the populations  $u(t,x) $ and $v(t,x)$ respectively.\\
$ \bullet $. The nonlinear terms $ g_{i}:\mathbf{R}\times \mathbf{R}\times \mathbf{R}\times \mathbf{R} \longrightarrow  \mathbf{R} $, $ i=1,2 $ are defined by $$ g_1 (t,u,v,\nabla v(t,x)) =- a(t) u(t,x) + c_1 (t)\dfrac{ v(t,x) u(t,x) }{1+| \nabla v(t,x)|}
$$  
and 
$$ g_2 (t,u,v,\nabla u(t,x)) =c_2 (t)\dfrac{ u(t,x) v(t,x)}{1+|\nabla u(t,x)|}
$$  
where $ c_{i} \in L^{1}_{loc}(\mathbb{R}, \mathbb{R}^{+}) $, $i=1,2$ are the interaction terms of $u(t,x) $ and $v(t,x)$ respectively.\\
\subsection*{The abstract formulation}  Consider the Banach space $X:=C_{0}(\overline{\Omega})\times C_{0}(\overline{\Omega}) ,$ equipped with the given norm:
$\| \begin{pmatrix}
    \varphi   \\
    \psi  
  \end{pmatrix} \|=\|\varphi\|_{\infty}+\|\psi\|_{\infty}$, where $C_{0}(\overline{\Omega})$ is the space of continuous functions $ \varphi : \overline{\Omega}\longleftrightarrow \mathbb{R} $ such that $\varphi|_{\partial \Omega} = 0$.
  We define the linear operators $ (A(t) , D(A(t))), \ t\in \mathbb{R}$, by 
\begin{equation}
 \left\{
    \begin{aligned}
     A(t)  
     &:= \begin{pmatrix}
  d_{1}(t)   \Delta    & 0\\
 0 &   d_{2}(t) \Delta -b(t) 
  \end{pmatrix} ,\\   
     D(A(t) )& =C_{0}^{2}(\overline{\Omega}) \times C_{0}^{2}(\overline{\Omega}):=D
    \end{aligned} \label{EqApp1Chapter1}
  \right. 
\end{equation} 
where  $C_{0}^{2}(\overline{\Omega}):=\left\{  \varphi \in C_{0}(\overline{\Omega}) \cap H_{0}^{1}\left(  \Omega \right)  :\  \Delta \varphi  \in C_{0}(\overline{\Omega})  \right\}$ . Let  $\alpha\in(1/2,1)$ and $X_\alpha$ be the real interpolation space between $X$ and $D(A(t))$ given by $X_\alpha^t=X_\alpha := C_{\alpha}(\overline{\Omega})\times C_{\alpha}(\overline{\Omega})$ where $C_\alpha (\overline{\Omega})= \{\varphi \in C^{2\alpha}(\overline{\Omega}): \varphi|_{\partial \Omega} = 0\}$
with the norm $ \|\cdot\|_{0,\alpha}:= \|\cdot\|_{\infty}+\|\nabla \cdot\|_{\infty}+[\cdot]_{\theta} $ where $  [\varphi]_{\theta}:= \sup_{x,y\in \overline{\Omega}, \ x\neq y} \dfrac{|\varphi(x)-\varphi(y) |}{|x-y |^{\theta}}$, $\theta=2\alpha -1 $. Hence, $\| \begin{pmatrix}
    \varphi   \\
    \psi  
  \end{pmatrix} \|_{\alpha}=\|\varphi\|_{0,\alpha}+\|\psi\|_{0,\alpha}$ defines a norm on $X_\alpha $.
It is clear that, $X_\alpha\hookrightarrow X$, see \cite{Lun} for more details about the spaces of H\"older continuous functions $C^{2\alpha}(\overline{\Omega})$. Therefore, the nonlinear term   $ f: \mathbb{R}\times  X_{\alpha}\longrightarrow X $ is defined by $$ f(t, \begin{pmatrix}
   \varphi  \\
    \psi 
  \end{pmatrix} )(x)= \begin{pmatrix}
 g_1 (t, \varphi (x),\psi (x),\nabla  \psi (x))  \\
  g_2(t,\varphi (x),\psi (x),\nabla \varphi(x))
  \end{pmatrix} .$$
Then, \eqref{Eq3Chapter1} takes the following abstract form that corresponds to equation \eqref{Eq_1}:
\begin{equation*}
     x'(t)= A(t) x(t) +f(t,x(t)), \quad   \; t \in \mathbb{R}.
\end{equation*}

In order to prove the existence and uniqueness of almost periodic solutions to \eqref{Eq_1}, we check forward our main hypotheses \textbf{(H1)-(H5)}. One can remark that the domains are independent of $t$ i.e., $D(A(t))=D $ which yields that the real interpolation spaces $X_\alpha$ are also indepedent of $t$ and  hypothesis \textbf{(H4)} hold. Now, we prove that $ (A(t),D(A(t))), \ t\in \mathbb{R} $ satisfy hypotheses \textbf{(H1)-(H3)}. Define the operators  $A_{1}(t)=d_{1}(t)   \Delta  $ and  $A_{2}(t)=d_{2}(t)   \Delta  $. By \cite{DaSi}, the operator $A_0:=-\Delta $ on $C_{0}(\overline{\Omega})$ is sectorial with constants $M \geq 1$ and $\theta  \in (\frac{\pi}{2},\pi )$ such that  
\begin{eqnarray}
\| \lambda R(\lambda, A_0 )  \|_{\mathcal{L}(X)} \leq M \quad \text{ for all }\,    \lambda \in \Sigma_{\theta}. \label{Resolvent estimate1}
\end{eqnarray}    
Then, using  \eqref{Resolvent estimate1} and by assumptions on $d_i$ we claim that 
$$ \| \lambda R(\lambda, d_{i}(t) A_0  )  \|=\| \frac{\lambda}{d_{i}(t)} R(\frac{\lambda}{d_{i}(t)},  A_0  )  \| \leq M  \quad \text{ for all }\,  t\in \mathbb{R}.$$
 Hence, for each $t\in \mathbb{R}$, $A_{i}(t)$ generates a bounded analytic semigroup $ (T^{i}_t (\tau))_{\tau\geq 0} $ (with uniform bound $M$ with respect to $t$) on $C_{0}(\overline{\Omega})$ such that \begin{equation}
\parallel T^{i}_{t}(\tau)\parallel\leq M e^{-\lambda_{1} \tau }\qquad\text{ for }
\tau \geq0 \label{Exp stab semigroups},
\end{equation}
where $\lambda_{1}:= \min \{\lambda : \ \lambda \in \sigma(A_0 ) \} > 0$ and $ \sigma(A_0 ) $ is the spectrum of $-\Delta$ in $H^{1}_{0}(\Omega)$ and $ M=e^{\lambda_1 |\Omega|^{2/N}}(4 \pi)^{-1}$, see \cite{Haraux} for more details. Moreover, 
$$ \sup_{t,s \in \mathbb{R}} \|A_{i}(t)A_{i}(s)^{-1}\| =\sup_{t,s \in \mathbb{R}} \dfrac{d_{i}(t)}{d_{i}(s)} <\infty .$$
Furthermore, by $$ \| A_{i}(t)A_{i}(s)^{-1} - I_X \| =d_{i}(s)^{-1}\left|d_{i}(t)-d_{i}(s)  \right| \leq L_{i} \left|d_{i}(t)-d_{i}(s)  \right|  \leq L_{i} |t-s |$$ where $L_{i}:= L_{0,i}\left( \inf_{s \in \mathbb{R}} d_{i}(s) \right)^{-1} $, where $L_{0,i}$ is the H\"{o}lder constants of $d_i $.  Thus, we obtain that, for each $i=1,2 $, $ (A_{i}(t))_{t\in \mathbb{R}} $ generates an evolution family $ (U_{i}(t,s))_{t\geq s} $ on $C_{0}(\overline{\Omega})$. In otherwise, by \eqref{Exp stab semigroups}, the semigroups $ (T^{i}_t (\tau))_{\tau\geq 0} $ are hyperbolic with projections $ P(t)=I_X $ and $Q(t)=0$, $t\in \mathbb{R}$ with $$ \| \tau A_{i}(t)T^{i}_t (\tau) x\| \leq M e^{-\lambda_{1} \tau } \quad\text{ for }
\tau > 0 .$$
In view of  and by taking $\psi(\sigma):= M \sigma^{-1} e^{-\lambda_{1} \sigma }$, $ \sigma >0$ yields that $ L_i \| \varphi \|_{1} := L_i M\lambda_1^{-1}  <1 . $ 
Then, each evolution family $ (U_{i}(t,s))_{t\geq s} $ is hyperbolic with the same projection $I_0$ and exponent $\lambda_1$ and associated Green functions $G_{i}(t,s)=U(t,s)$, $ t\geq s $. Hence, the family of matrix-valued operators $ \begin{pmatrix}
  A_{1}(t)   & 0\\
 0 &   A_{2}(t)
  \end{pmatrix}_{t\in \mathbb{R}}  $ generates the hyperbolic evolution family $\left(  V(t,s)= \begin{pmatrix}
  U_{1}(t,s)   & 0\\
 0 &   U_{2}(t,s)
  \end{pmatrix} \right)_{t\geq s} $ with projection $\left( P(t)= \begin{pmatrix}
  I_0  & 0\\
 0 &   I_0
  \end{pmatrix}\right)_{t\in \mathbb{R}}  $ and exponent $\lambda_1$. Moreover, by rescaling,  we obtain that $(A(t) )_{t\in \mathbb{R}} $ generates the hyperbolic evolution family $$\left(  U(t,s)= \begin{pmatrix}
  U_{1}(t,s)   & 0\\
 0 &  e^{- \int_{s}^{t} b(\sigma)d\sigma } U_{2}(t,s)
  \end{pmatrix} \right)_{t\geq s} $$ with projection $\left( P(t)= \begin{pmatrix}
  I_0  & 0\\
 0 &   I_0
  \end{pmatrix}\right)_{t\in \mathbb{R}}  $ and exponent $\delta= \lambda_1$. Thus, hypotheses \textbf{(H1)} and \textbf{(H2)} hold with Green function $G(t,s):=U(t,s), \ t \geq s$. To check hypothesis \textbf{(H3)}, we need the following preliminary result.
  \begin{lemma}
 For each $i=1,2$,  if $ A_{i} (\cdot) \in AP(\mathbb{R}, \mathcal{L}(C_{0}^{2}(\overline{\Omega}), X))$. Then, the associated evolution family $U_{i}(\cdot,\cdot) $ is bi-almost periodic.
\end{lemma}\label{Lemma1 App}
\begin{proof}
Let $ i=1,2 $ be fixed and $t, \tau \in \mathbb{R}  $. Then, we have
\begin{eqnarray}
A_{i}(t +\tau )^{-1}- A_{i}(t)^{-1}= A_{i}(t +\tau )^{-1}(A_{i}(t +\tau )-A_{i}(t) )A_{i}(t)^{-1} \label{AP Green function form1}
\end{eqnarray}
Following \cite[Theorem 4.5.]{Man-Sch}, it suffices to show that $ A^{-1}_{i} (\cdot) \in AP(\mathbb{R}, \mathcal{L}(X))$. Let $ \varepsilon>0 $ and $ f \in X $, from \eqref{AP Green function form1} and $ A_{i} (\cdot) \in AP(\mathbb{R}, \mathcal{L}(C_{0}^{2}(\overline{\Omega}), X))$, it follows that there exists $l_\varepsilon>0$ such that every interval of length $l_\varepsilon$ contains an element $\tau$ such that
\begin{eqnarray*}
\| A_{i}(t +\tau )^{-1}f- A_{i}(t)^{-1}f \|&=&\| A_{i}(t +\tau )^{-1}(A_{i}(t +\tau )-A_{i}(t) )A_{i}(t)^{-1}f \| \\
&\leq & \| A_{i}(t +\tau )^{-1}\|_{\mathcal{L}(X)} \| A_{i}(t +\tau )-A_{i}(t)\|_{\mathcal{L}(D, X)} \| A_{i}(t)^{-1}f \|_{C_{0}^{2}(\overline{\Omega})} \\
&\leq & C \varepsilon .
\end{eqnarray*}
\end{proof}
Consequently, we have the following main result.
\begin{proposition}3
Let $d_i \in AP(\mathbb{R}),$ $i=1,2$ and $b \in APS^1(\mathbb{R})$. Then, for each $ \begin{pmatrix}
   \varphi \\
    \psi 
  \end{pmatrix} \in X$,  the Green function $G(\cdot,\cdot) \begin{pmatrix}
   \varphi \\
    \psi 
  \end{pmatrix}  $ is bi-almost periodic. Hence, hypothesis \textbf{(H3)} is satisfied.
\end{proposition} 
\begin{proof}
Since $d_i \in AP(\mathbb{R})$, it follows that $ A_{i} (\cdot) \in AP(\mathbb{R}, \mathcal{L}(C_{0}^{2}(\overline{\Omega}), X))$ for $i=1,2$. Then, by Lemma \ref{Lemma1 App}, we obtain that $U_{i}(\cdot,\cdot) $ are bi-almost periodic. Now, we show that $ e^{-\displaystyle \int_{\cdot}^{\cdot} b(\sigma)d\sigma} U_{2}(\cdot,\cdot)$ is almost periodic using Theorems \ref{Theorem charac AP with sequences} and \ref{Theorem charac APSp with sequences}.  Let $(\sigma_n)_n $ be any sequence of real numbers, by $b \in APS^1(\mathbb{R})$, it holds that there exist a subsequence $(\tau_n)_n \subset (\sigma_n)_n  $ and functions $ \tilde{b} $ and $  \tilde{U}_{i}(\cdot,\cdot) $ such that  \eqref{charac AP with sequences} and \eqref{charac APSp with sequences} hold for $ b $ and $ U_{i}(\cdot,\cdot) $ respectively. Define the function $ e^{-\displaystyle \int_{\cdot}^{\cdot} \tilde{b}(\sigma)d\sigma} \tilde{U}_{2}(\cdot,\cdot) $. Then, we obtain that
\begin{eqnarray*}
 & &\|  e^{-\displaystyle  \int_{s+\tau_n}^{t+\tau_n} b(\sigma)d\sigma } U_{2}(t+\tau_n,s+\tau_n)-e^{-\displaystyle \int_{s}^{t} \tilde{b}(\sigma)d\sigma } \tilde{U}_{2}(t,s) \| \\
 &\leq  & e^{-\displaystyle \int_{s}^{t} \tilde{b}(\sigma)d\sigma }  \| U_{2}(t+\tau_n,s+\tau_n) \| \left| e^{-\displaystyle \int_{s}^{t}\left[  b(\sigma+\tau_n)-\tilde{b}(\sigma) \right] d\sigma }-1 \right| \\ 
 &+& e^{- \displaystyle\int_{s}^{t} \tilde{b}(\sigma)d\sigma } \|  U_{2}(t+\tau_n,s+\tau_n) -\tilde{U}_{2}(t,s)\| \\ 
 &\leq  &  M \left| e^{-\displaystyle \int_{s}^{t}\left[  b(\sigma+\tau_n)-\tilde{b}(\sigma) \right] d\sigma }-1 \right| +  \|  U_{2}(t+\tau_n,s+\tau_n) -\tilde{U}_{2}(t,s)\|.
 \end{eqnarray*}
 Therefore, we have $  \|  U_{2}(t+\tau_n,s+\tau_n) -\tilde{U}_{2}(t,s)\| \rightarrow 0 $ as $ n\rightarrow \infty $ and
 \begin{eqnarray*}
 e^{-\displaystyle \int_{s}^{t}\left[  b(\sigma+\tau_n)-\tilde{b}(\sigma) \right] d\sigma } &\leq &  e^{-\displaystyle \sum_{k=[s]}^{[t]} \int_{k}^{k+1} | b(\sigma+\tau_n)-\tilde{b}(\sigma) | d\sigma }  \\ 
 &= & e^{-([t]-[s])}e^{-\displaystyle \sup_{k} \int_{k}^{k+1} | b(\sigma+\tau_n)-\tilde{b}(\sigma) | d\sigma } \\
 &\leq & e^{-\displaystyle \sup_{k} \int_{k}^{k+1} | b(\sigma+\tau_n)-\tilde{b}(\sigma) | d\sigma } \rightarrow 1 \text{ as } n\rightarrow \infty,
 \end{eqnarray*}
 uniformly in $t,s \in \mathbb{R}$, $t\geq s$.
 Thus,  $$ \|  e^{-\displaystyle  \int_{s+\tau_n}^{t+\tau_n} b(\sigma)d\sigma } U_{2}(t+\tau_n,s+\tau_n)-e^{-\displaystyle \int_{s}^{t} \tilde{b}(\sigma)d\sigma } \tilde{U}_{2}(t,s) \| \rightarrow 0  \text{ as }  n\rightarrow \infty, $$  uniformly in $t,s \in \mathbb{R}$, $t\geq s$. Hence $ e^{-\displaystyle \int_{\cdot}^{\cdot} b(\sigma)d\sigma} U_{2}(\cdot,\cdot) $ is bi-almost periodic. Consequently, if we consider $ \tilde{G}(\cdot, \cdot)=\begin{pmatrix}
  \tilde{U}_{1}(\cdot,\cdot)   & 0\\
 0 &  e^{- \int_{\cdot}^{\cdot} \tilde{b}(\sigma)d\sigma } \tilde{U}_{2}(\cdot,\cdot)
  \end{pmatrix} $. Then, for  $\begin{pmatrix}
   \varphi \\
    \psi 
  \end{pmatrix} \in X$, we obtain that
\begin{eqnarray*} 
&&\| G(t+\tau_n ,s+\tau_n)  \begin{pmatrix}
   \varphi \\
    \psi 
  \end{pmatrix} -  \tilde{G}(t,s) \begin{pmatrix}
   \varphi \\
    \psi 
  \end{pmatrix} \| \leq \|  U_{1}(t+\tau_n,s+\tau_n) -\tilde{U}_{1}(t,s)\| \| \varphi\|_{\infty} \\&+& \| e^{-\displaystyle  \int_{s+\tau_n}^{t+\tau_n} b(\sigma)d\sigma } U_{2}(t+\tau_n,s+\tau_n) -e^{-\displaystyle \int_{s}^{t} \tilde{b}(\sigma)d\sigma}\tilde{U}_{2}(t,s)\| \| \psi\|_{\infty} \\
   &\leq & (\| e^{-\displaystyle  \int_{s+\tau_n}^{t+\tau_n} b(\sigma)d\sigma } U_{2}(t+\tau_n,s+\tau_n) -e^{-\displaystyle \int_{s}^{t} \tilde{b}(\sigma)d\sigma}\tilde{U}_{2}(t,s)\|   \\ &+& \|  U_{1}(t+\tau_n,s+\tau_n) -\tilde{U}_{1}(t,s)\| )  \begin{pmatrix}
   \varphi \\
    \psi 
  \end{pmatrix} \rightarrow 0 \text{ as } n\rightarrow \infty , 
\end{eqnarray*}
uniformly in $t,s \in \mathbb{R}$, $t\geq s$. This proves the result.
\end{proof}
\begin{proposition}\label{Proposition Local Lipsch}
The function $f $ satisfies \textbf{(H5)}.
\end{proposition}
\begin{proof}
Let $ \begin{pmatrix}
   \varphi_1 \\
    \psi_1
  \end{pmatrix}, \begin{pmatrix}
   \varphi_2  \\
    \psi_2 
  \end{pmatrix} \in X_{\alpha} $ and $\rho >0$ be such that $  \| \begin{pmatrix}
   \varphi_1 \\
    \psi_1
  \end{pmatrix}\|_{\alpha}, \ \| \begin{pmatrix}
   \varphi_2  \\
    \psi_2 
  \end{pmatrix} \|_{\alpha} \leq \rho $. Then, we have
\begin{eqnarray*}
& & | g_1 (t, \varphi (x),\psi (x),\nabla  \psi (x)) -g_1 (t, \varphi (x),\psi (x),\nabla  \psi (x)) | \\ 
&=&  \left| a (t)\left[\varphi_{1}(x) - \varphi_{2}(x) \right]-  c_1 (t)  \left[ \dfrac{ \varphi_{1}(x)  \psi_{1}(x) }{1+| \nabla  \psi_{1}(x)|}  - \dfrac{\varphi_{2}(x) \psi_{2}(x) }{1+| \nabla  \psi_{2}(x)|} \right] \right| \\ 
 &=&  \left| a (t)\left[\varphi_{1}(x) - \varphi_{2}(x) \right]-  c_1 (t)  \left[ \dfrac{\varphi_{1}(x)  \psi_{1}(x)-\varphi_{2}(x) \psi_{2}(x) }{1+| \nabla  \psi_{2}(x)|} \right] \right.  \\ &-&  \left.  c_1 (t)  \left[ \dfrac{ \varphi_{1}(x)  \psi_{1}(x) (| \nabla  \psi_{2}(x)|-| \nabla  \psi_{1}(x)|)}{(1+| \nabla  \psi_{1}(x)|)(1+| \nabla  \psi_{2}(x)|)}\right]   \right|  \\ 
  &=&  \left| \dfrac{a (t)-c_1 (t) \psi_{2}(x)}{1+| \nabla  \psi_{2}(x)|} \left[\varphi_{1}(x) - \varphi_{2}(x) \right]-\dfrac{ c_1 (t) \varphi_{1}(x)}{1+| \nabla  \psi_{2}(x)|} \left[\psi_{1}(x) - \psi_{2}(x) \right] \right. \\ 
  &-& \left.  \dfrac{ c_1 (t) \varphi_{1}(x)  \psi_{1}(x) }{(1+| \nabla  \psi_{1}(x)|)(1+| \nabla  \psi_{2}(x)|)}(| \nabla  \psi_{2}(x)|-| \nabla  \psi_{1}(x)|)   \right|  \\ 
   \\ &\leq & (a (t) + c_1 (t)\rho)   \| \varphi_1 -\varphi_2 \|_{0,\alpha}+c_1 (t) \rho (\rho +1)  \| \psi_1 -\psi_2 \|_{0,\alpha} \\
  &\leq  & (a (t) + c_1 (t) \rho (\rho +1) )  \| \begin{pmatrix}
   \varphi_1 \\
    \psi_1
  \end{pmatrix}-\begin{pmatrix}
   \varphi_2 \\
    \psi_2
  \end{pmatrix}\|_{\alpha}, \quad x\in \overline{\Omega}, \ t\in \mathbb{R}.
\end{eqnarray*} 
Arguing as above, we obtain that 
 \begin{eqnarray*}
 | g_2(t,\varphi_{1} (x),\psi_{1} (x),\nabla \varphi_{1}(x))-g_2(t,\varphi_{2} (x),\psi_{2} (x),\nabla \varphi_{2}(x))| 
  &\leq &  c_2 (t) \rho (\rho +1)   \| \begin{pmatrix}
   \varphi_1 \\
    \psi_1
  \end{pmatrix}-\begin{pmatrix}
   \varphi_2 \\
    \psi_2
  \end{pmatrix}\|_{\alpha},
\end{eqnarray*}  
for $ x\in \overline{\Omega}, \ t\in \mathbb{R}$. Hence, we obtain that
\begin{eqnarray*}
\| f(t, \begin{pmatrix}
   \varphi_1  \\
    \psi_1
  \end{pmatrix} )-f(t, \begin{pmatrix}
   \varphi_2  \\
    \psi_2 
  \end{pmatrix} )\| \leq    \left( a(t)+ (c_1 (t)+c_2 (t)) \rho (\rho +1) \right)   \| \begin{pmatrix}
   \varphi_1 \\
    \psi_1
  \end{pmatrix}-\begin{pmatrix}
   \varphi_2 \\
    \psi_2
  \end{pmatrix}\|_{\alpha}, \;  \; t\in \mathbb{R.}
\end{eqnarray*}
Therefore, $ f $ satisfies \textbf{(H5)} with $ L_{\rho}(t)=  a(t)+ (c_1 (t)+c_2 (t)) \rho (\rho +1) .$ 
\end{proof}
\begin{proposition}
Assume that for each $ i=1,2 $, $ c_{i} \in APS^{1}(\mathbb{R})$. Then, $f \in APS^{1}U(\mathbb{R} \times X_{\alpha},X)$.
\end{proposition}
\begin{proof}
By assumptions for each $ \begin{pmatrix}
   \varphi \\
    \psi
  \end{pmatrix} \in X_{\alpha} $, we have that $ f(\cdot, \begin{pmatrix}
   \varphi_2 \\
    \psi_2
  \end{pmatrix}) $ is $ APS^{1}(\mathbb{R},X) $. Moreover, by Proposition \ref{Proposition Local Lipsch}, the condition \eqref{ComCharSpChapter1} holds for $f$. Hence, we conclude by Lemma \ref{LemApSpCom1Chapter1}.
\end{proof}
\subsection*{Main result} Let us define the following functions.
 
 $$ d_i (t)= \tilde{d}_{i} + \hat{d}_{i}( \cos(t)+\cos(\pi t) ), \; t \in \mathbb{R} \; \text{ for } i=1,2,$$
 where $ \tilde{d}_{i} >2\hat{d}_{i} $ and $\dfrac{\hat{d}_{i}}{\tilde{d}_{i}-2\hat{d}_{i}} < \dfrac{4 \pi \lambda_1 }{(1+\pi)e^{-\lambda_1 |\Omega |^{2/N}}}$. Moreover, for $ \tilde{a}, \tilde{b},\tilde{c} \geq 0 $, we define 
$$a(t)=\begin{array}{cc}
\left\{
\begin{array}{ccc}
\tilde{a}(2+ \cos (t)+\cos \left(\sqrt{2} t\right) ),& t\geq 0, \\
\tilde{a}(2+ \sin (t)+\sin \left(\sqrt{2} t\right)),& t<0,
\end{array}
\right.
 \\
\end{array}$$
\begin{equation*} 
b(t)=\left\{
\begin{aligned}\tilde{b}(1+ \sin(t)), &\qquad  2k \pi  \leq t < (2k+1) \pi, \; k\in \mathbb{Z} , \\
\tilde{b}(1+ \cos(t)), &\qquad  (2k+1) \pi \leq t \leq (2k+2) \pi,  \; k\in \mathbb{Z} , \\
\end{aligned} \right.
\end{equation*}
and
$$ c_{1}(t)= \tilde{c}\sin(1/p(t))   \text{ and }  c_{2}(t)=\tilde{c} \cos(1/p(t))  \text{  where } p(t)=2+\sin(t)+\sin(\sqrt{2}t), \; t\in \mathbb{R}.$$
\begin{proposition}\label{PropApp3} The following assertions are true:\\
\textbf{(i)} For each $ i=1,2 $, the functions $d_{i} \in AP(\mathbb{R}) $ are such that $ \inf_{t\in \mathbb{R}} | d_{i}(t) | =  \tilde{d}_{i} -2\hat{d}_{i} > 0$  and $  \sup_{t\in \mathbb{R}}  d_{i}(t)  =  \tilde{d}_{i} +2\hat{d}_{i} < \infty$. Moreover,  $d_{i} $ are $\hat{d}_{i}(1+\pi)  $-Lipschitzian i.e. H\"olderian with exponent $\mu=1$.\\
\textbf{(ii)}  The functions $a, b \in APS^{1}(\mathbb{R}) \cap L^{\infty}(\mathbb{R})   $, but, $ a, b \notin AP(\mathbb{R}) $ since they are discontinuous in $\mathbb{R}$. However,  for all $t\in \mathbb{R}$, $ 0 \leq a(t) \leq   \sup_{t\in \mathbb{R}} a (t) =4\tilde{a} $ and $  0 \leq b(t) \leq   \sup_{t\in \mathbb{R}}  b (t) =2 \tilde{b} $.\\
\textbf{(iii)} For each $ i=1,2 $, the function $ c_{i} \in  APS^{1}(\mathbb{R}) \cap AA(\mathbb{R}) $, but, $ c_{i} \notin AP(\mathbb{R}) $ since it is not uniformly continuous. Moreover, for all $ t\in \mathbb{R}$, $ 0 \leq c_i (t) \leq   \sup_{t\in \mathbb{R}} c_i (t) =\tilde{c} $. 
\end{proposition}
\begin{proof} \textbf{(i)}  The function $ t\longmapsto   \cos(t)+\cos(\pi t)$ is almost periodic in Bohr sense since it is in particular the sum of two periodic functions with the ratio of the two periods $\dfrac{2\pi}{2}=\pi \notin \mathbb{Q}$. Then, in view of \cite{H.Bohr}, $ t\longmapsto   \cos(t)+\cos(\pi t)$ is almost periodic not periodic. Hence, the other facts now are obvious.   \\
\textbf{(ii)} The functions $ a,b $ are piecewise defined by almost periodic functions in Stepanov sense. Hence, using Definition \ref{Definition APSp}, we can prove easily that the entire function belongs to $  APS^{1}(\mathbb{R})  $. Moreover, it is clear that $ a $ and $b$ are bounded and discontinuous where $ a $ is discontinuous at $t=0$ and $b$ is discontinuous in $ 2\mathbb{Z} +1 $.  \\
\textbf{(iii)} See \cite[Example 3.1]{BaGu}.
\end{proof}
\begin{theorem} 
Assume that $\dfrac{\hat{d}_{i}}{\tilde{d}_{i}-2\hat{d}_{i}} < \dfrac{4 \pi \lambda_1 }{(1+\pi)e^{\lambda_1 |\Omega |^{2/N}}}$ and $\tilde{a} < \pi \lambda_1  e^{-\lambda_1 |\Omega |^{2/N}} c_{\alpha}^{-1}$ for some constant $ c_{\alpha} >0$. Then, there exists $ \rho$ which satisfies
 $$0 \leq \rho <\left( -\tilde{c}+ \sqrt{4 \pi  \lambda_1 e^{-\lambda_1 |\Omega |^{2/N}} c_{\alpha}^{-1} -4\tilde{a}+\tilde{c}^{2}} \right) (2\tilde{c}^{-1})$$
such that our system \eqref{Eq3Chapter1} has a unique almost periodic mild solution that satisfies  $ \sup_{x\in \overline{\Omega}}|u(t,x)| $, $\sup_{x\in \overline{\Omega}}|v(t,x)| \leq \rho $ for all $ t\in \mathbb{R} $.
\end{theorem}
\begin{proof}
Under the above considerations, hypotheses \textbf{(H1)-(H5)} hold. Furthermore, using Proposition \ref{PropApp3}, for each $ \begin{pmatrix}
   \varphi \\
    \psi
  \end{pmatrix} \in X_{\alpha} $, we have that $ f(\cdot, \begin{pmatrix}
   \varphi_2 \\
    \psi_2
  \end{pmatrix}) $ is $ APS^{1}(\mathbb{R},X) \cap L^{\infty}(\mathbb{R},X) $. Moreover, in view of Proposition \ref{PropApp3}, we obtain that $ L_{\rho} \in L^{\infty}(\mathbb{R})$ with $\| L_{\rho}\|_{\infty}=  4\tilde{a}+ 2\tilde{c} \rho (\rho +1) .$ Hence, we conclude using Theorem \ref{Main Theorem2 Eq1}. Indeed, consider the following algebraic equation: 
\begin{eqnarray*}
2\tilde{c} \rho^{2}+ 2\tilde{c} \rho +4\tilde{a}- \dfrac{4\pi \lambda_{1}}{c_{\alpha} e^{-\lambda_1 |\Omega |^{2/N}}}
\end{eqnarray*}
which admits two solutions $\rho_{0}= \left( -\tilde{c}- \sqrt{4 \pi  \lambda_1 e^{-\lambda_1 |\Omega |^{2/N}} c_{\alpha}^{-1} -4\tilde{a}+\tilde{c}^{2}} \right) (2\tilde{c})^{-1} \leq 0$ and $\rho_{1}= \left( -\tilde{c}- \sqrt{4 \pi  \lambda_1 e^{-\lambda_1 |\Omega |^{2/N}} c_{\alpha}^{-1} -4\tilde{a}+\tilde{c}^{2}} \right) (2\tilde{c})^{-1} \geq 0$ where $ c_{\alpha} $ is the constant from interpolation between $ X_{\alpha}$ and $X$. Then, by taking $0 \leq  \rho < \rho_1 $, condition  \eqref{Contra cond Theorem2} holds. Thus equation  \eqref{Eq_1} has a unique mild solution $ u \in AP(\mathbb{R},X_{\alpha}) $ with $ \| u(t) \| \leq \rho $ for all $ t\in \mathbb{R} $, which is equivalent to the existence and uniqueness of an almost periodic mild solution to the system \eqref{Eq3Chapter1} such that $ \sup_{x\in \overline{\Omega}}|u(t,x)| $, $\sup_{x\in \overline{\Omega}}|v(t,x)| \leq \rho $ for all $ t\in \mathbb{R} $. This proves the result.

\end{proof}

\end{document}